\documentclass[]{article}

\usepackage{hyperref}
\usepackage{amsmath,amsthm}%,indentfirst}
\usepackage[]{graphicx}
\usepackage{amsfonts}
 \usepackage{epsfig}
 \usepackage{epstopdf}
\usepackage{float}
\usepackage{calrsfs}
\usepackage{subfigure}
\newcommand{\R}{{\mathbb{R}}}
\newcommand{\N}{{\mathbb{N}}}
\newcommand{\cH}{{\mathcal{H}}}

\newtheorem{theorem}{Theorem}
\newtheorem{lemma}{Lemma}
\newtheorem{definition}{Definition}
\newtheorem{corollary}{Corollary}
\title{LARGE TIME BEHAVIOR OF A COMPLEX NETWORK OF REACTION-DIFFUSION SYSTEMS OF FITZHUGH-NAGUMO TYPE}

\author{ B. Ambrosio, M.A. Aziz-Alaoui, V.L.E. Phan\footnote{Normandie Univ, France; ULH, LMAH, F-76600 Le Havre; FR CNRS 3335, ISCN, 25 rue
Philippe Lebon 76600 Le Havre, France. Mail:
benjamin.ambrosio@univ-lehavre.fr, aziz.alaoui@univ-lehavre.fr, pvlem6a2@gmail.com ).}}

\begin{document}
\maketitle

\begin{abstract}
We focus on the long time behavior of complex networks of reaction-diffusion (RD) systems. We prove the existence of the global attractor and a $L^{\infty}$-bound for a network of $n$ RD systems with $d$ variables each. This allows us to prove  the identical synchronization for general class of networks and establish the existence of a coupling strength threshold value that ensures such a synchronization. Then, we apply these results to some particular networks with different structures (i.e. different topologies) and perform numerical simulations. We found out theoretical and numerical heuristic laws for the minimal coupling strength  needed for synchronization   relatively to the number of nodes and the network topology, and discuss the link between  spatial dimension and synchronization.
\end{abstract}

Reaction-Diffusion systems, complex networks, attractor,  synchronization.

\pagestyle{myheadings}
\thispagestyle{plain}
\markboth{B. Ambrosio, M.A. Aziz-Alaoui, V.L.E Phan}{Complex network of reaction-diffusion equations}

\section{Introduction}

Networks of dynamical systems appear naturally in the modeling of  numerous applications. The simplest mode for the coordinated motion between  dynamical systems is their identical synchronization when all nodes of the network acquire identical dynamical behavior. Such cooperative behavior has been observed in natural or artificial systems such as neural networks, chemical and biological systems, computer clocks, social networks... In this paper, we will focus on RD systems networks that can be seen as neural networks. Besides, a classical question in dynamical systems is the existence of the attractor: basically, a set that attracts the trajectories for large time.

\textbf{Contributions}. There are three main contributions in the present paper.
First, we prove of the existence of the network attractor, and therefore, within this attractor, we analyze the synchronization behavior. Finally, we perform numerical simulations with different kind
of initial conditions and study numerically the spatial effects on synchronization and pattern formation.

\textbf{Mathematical framework and preliminaries.}
First of all, we introduce the mathematical framework we will use throughout this paper. Mathematically speaking, the network is represented by a graph,
the nodes of which are a d-dimensional RD system and the edges correspond to the coupling functions between these subsystems.
The general system reads as:
\begin{equation}\label{eq:Reseau}
{U_{it}} = \tilde{Q}\Delta {U_i}+\tilde{F}({U_i})  + \tilde{H}_i(U_1,...,U_n), i\in \{1,...,n\}.
\end{equation}
In this equation, each variable $U_i$ represents a function from $\Omega \times \R^+$ into  $\R^d$,  $\Omega$ is a bounded domain of $\R^N$ and $\tilde{F}: \R^d \rightarrow \R^d$
is the nonlinear reaction term.
For all $i\in \{1,...,n\}$, $\tilde{H}_i:\R^{nd} \rightarrow \R^d$ is the coupling function between nodes whereas $\tilde{Q}$ is a diagonal matrix of $\R^{d\times d}$ with positive coefficients.
If we add boundary conditions to \eqref{eq:Reseau}, we obtain a general reaction-diffusion system.
We will not go into details concerning the existence of the semi-group of \eqref{eq:Reseau}. We refer  to \cite{Ladyzens,JLL,Robinson, Temam} or
\cite{Friedman2, Henry, Mora, Rothe,Smoller}, for classical results on the existence of  semi-group in $L^p(\Omega)$ or in $C^{k,\alpha}(\Omega)$ spaces.
Our first theoretical result is the proof  of the existence of the global attractor for a particular class of networks of type  \eqref{eq:Reseau}
  that generalizes the FitzHugh-Nagumo (FHN) equations. Recall that (FHN) equations are a simplification in two variables of the HodgKin-Huxley model of four equations for the action propagation  in nerve,
  \cite{Fitz,HH,N}. A good qualitative analysis of the (FHN) reaction-diffusion system is given in \cite{Smoller2}, while in \cite{Ambrosio1} we gave a first analysis of a particular network of (FHN)
  reaction-diffusion systems. Here we extend some results of   \cite{Ambrosio1}, and also of \cite{Marion} where a system of two scalar equations,
  in which the diffusive term appears only in the first equation but not in the second one, was considered. We present some results for a network of $n$ partially diffusive systems with $d$ equations.
Indeed, we suppose that we can split the  system \eqref{eq:Reseau} into two subsystems, diffusive and non-diffusive, with $s$  and $d-s$ equations.
Therefore, we set  for all $i \in \{1,...,n\}$, $U_i=(u_i,v_i)$, and write \eqref{eq:Reseau}  in the following way:
\begin{equation}
\hspace{-0.5cm}
\label{eq:equv}
\left\{ \begin{array}{rcl}
  u_{it} &=& F({u_i,v_i}) + Q\Delta {u_i} + {H_i}(u_1,...,u_n), \, \,\mbox{ on } \Omega\times ]0,+\infty[, \qquad i\in \{1,...,n\}\\
 v_{it} &=&  - \sigma (x)v_i + \Phi (x,u_i) \mbox{ on } \Omega\times ]0,+\infty[,
 \end{array} \right.
\end{equation}
with Neumann Boundary conditions on $\partial \Omega$,  and
 where $u_i$ take values in $\R^s$, $1\leq s <d$  whereas $v_i$ take values in $\R^{d-s}$,  $Q$ is a diagonal matrix in $\R^{s\times s}$
 with coefficients $q^j, j \in \{1,...,s\}$, and  $H_i$ take values in $\R^s$.
  We use the classical notation $u_t$ for $\frac{\partial u}{\partial t}$.
 This means that diffusion and  coupling terms appear only in the $s$ first variables of each subsystem of the network.
 Finally, $\sigma(x)$ is a matrix in $\R^{(d-s)\times (d-s)}$ that verifies: $\sum_{j=1}^{d-s}\sum_{l=1}^{d-s}\sigma_{jl}(x) v_lv_j^{k} >\sigma \sum_{j=1}^{d-s}v_j^{k+1}$, for all $k \in \N$, $k$ odd, and some positive constant $\sigma$,
 and with bounded derivatives.
 The application $\Phi$ takes values in $\R^{(d-s)}$. Under some conditions on functions the system \eqref{eq:equv} generates a semi-group on $\mathcal{H}=(L^2(\Omega))^{nd}$,
 see \cite{Marion, Robinson,Temam}.
Before going into details of the analysis of system \eqref{eq:equv}, we present some key features for systems with one variable and two variables used to prove the existence of the global attractor.
These  techniques will be generalized to system \eqref{eq:equv} in section \ref{section:exofGA}. Let us start with the following equation:
\begin{equation}
\label{eq:1d}
u_t=\Delta u -u^3+u^2+u,
\end{equation}
considered in a bounded domain $\Omega$ with Neumann boundary conditions.
Multiplying \eqref{eq:1d} by $u$ gives,
\begin{align}
\frac{d}{dt}\int_\Omega u^2+2\int_\Omega|\nabla u|^2=&  -\int_\Omega (u^4-u^3-u^2) \label{eq:1dL2Bound-a}\\
\leq & -\delta \int_\Omega u^2 +K,\label{eq:1dL2Bound-b}
\end{align}
for some constants $\delta$ and $K$.
By Gronwall Lemma, there exists a constant $K'$ such that,
\begin{equation*}
\int_\Omega u^2\leq K',
\end{equation*}
 for all initial conditions in $L^2(\Omega)$ and for $t$ large enough.
Now, integrating \eqref{eq:1dL2Bound-b} between $t$ and $t+r$ for a given constant $r$ gives,
 \begin{equation}
 \label{eq:1dL2Bound-c}
\int_t^{t+r}\int_\Omega |\nabla u|^2\leq K \mbox{ for another } K.
\end{equation}

Multiplying \eqref{eq:1d} by $u^{2k-1}$, by analog computations, we find that there exists a constant $K''$ such that
\begin{equation}
 \label{eq:1dLkBound}
\int_\Omega u^{2k}\leq K'',
\end{equation}
 for all initial conditions in $L^2(\Omega)$ and for $t$ large enough.\\
Also, multiplying \eqref{eq:1d} by $-\Delta u$ gives,
\begin{align*}
\frac{d}{dt}\int_\Omega |\nabla u|^2=& -2\int_\Omega(\Delta u)^2 +2\int_\Omega (u^4-u^3-u^2)\Delta u\\
\leq & -\int_\Omega(\Delta u)^2 +\frac{3}{2}\int_\Omega (u^8+u^6+u^4) \mbox{ by using Young inequality}.
\end{align*}
Therefore, thanks to \eqref{eq:1dL2Bound-c} and  \eqref{eq:1dLkBound}, we deduce, by using uniform Gronwall Lemma (see appendix),  that
\begin{equation*}
\int_\Omega |\nabla u|^2<K,
\end{equation*}
for a constant $K$ and for all initial conditions in $L^2(\Omega)$ for $t$ large enough.\\
This gives the compacity of trajectories of \eqref{eq:1d}  thanks to the compact injection of $H^1$ in $L^2$.
Now, we consider the  system with two variables,
\begin{equation}
\label{eq:2d}
\left\{
\begin{array}{rcl}
 u_t&=&\Delta u -u^3+u^2+u +v\\
 v_t&=&-\delta v +u
\end{array}
\right.
\end{equation}
Multiplying the first equation of \eqref{eq:2d} by $u$ and the second by $v$, integrating, using Green formula, Young inequality, and Gronwall lemma leads to:
\begin{equation*}
\int_\Omega(u^2+v^2)<K,
\end{equation*}
for a constant $K$, for all initial conditions in $L^2(\Omega) \times L^2(\Omega)$, and time large enough.
Then, we will show the same result in $L^{2k}(\Omega)\times L^{2k}(\Omega)$ for all $k\in \N^*$. As we saw above, this result is true for $k=1$.
We multiply first equation of \eqref{eq:2d} by $u^{2k-1}$ and the second by $v^{2k-1}$, sum the two equations and integrate. We obtain:
\begin{equation*}
\frac{d}{dt}\int_\Omega(u^{2k}+v^{2k})=-\int_\Omega u^{2k-2}|\nabla u|^2-\int_\Omega (u^{2k+2}+u^{2k+1}+u^{2k})+\int_\Omega vu^{2k-1}+\int_\Omega uv^{2k-1}-\delta\int_\Omega v^{2k}.
\end{equation*}
Using Young inequality $ ab\leq \frac{c^pa^p}{p}+\frac{b^q}{c^qq}$, provides:
\begin{equation*}
\frac{d}{dt}\int_\Omega(u^{2k}+v^{2k})\leq-\int_\Omega (u^{2k+2}+u^{2k+1}+u^{2k})+\gamma_1\int_\Omega v^{\frac{2k+2}{3}}+ \gamma_2\int_\Omega u^{2k+2}+\gamma_3\int_\Omega u^{2k}+\gamma_4\int_\Omega v^{2k}-\delta\int_\Omega v^{2k},
\end{equation*}
with $\gamma_2<1$ and $\gamma_4<\delta$. Then, since $\frac{2k+2}{3}<2k$, there exists constants $\gamma$ and $K$ such that:
\begin{equation*}
\frac{d}{dt}\int_\Omega(u^{2k}+v^{2k})\leq-\int_\Omega (u^{2k}+v^{2k})+K.
\end{equation*}
It follows that there exists a constant $K$ (depending on $k$) such that:
\begin{equation*}
\int_\Omega(u^{2k}+v^{2k})<K,
\end{equation*}
 for all initial conditions in $L^2(\Omega) \times L^2(\Omega)$, and time large enough.
 By analog computations as for \eqref{eq:1d}, we obtain:
 \begin{equation*}
\int_\Omega |\nabla u|^2<K.
\end{equation*}
It  remains to consider $\int_\Omega |\nabla v|^2$. By multiplying gradient of the second equation of \eqref{eq:2d} by $\nabla v$, we obtain:
\begin{equation*}
\label{eq:vd}
 \frac{d}{dt}\int_\Omega |\nabla v|^2=-2\delta |\nabla v|^2 +2\nabla u.\nabla v,
\end{equation*}
it follows that,
\begin{equation*}
 \frac{d}{dt}\int_\Omega |\nabla v|^2\leq-\delta |\nabla v|^2 +K,
\end{equation*}
which gives,
 \begin{equation*}
\int_\Omega |\nabla v|^2<K,
\end{equation*}
for $t$ large enough. We have bounds in $L^q$ for all $q \in [1,+\infty[$. We can obtain bounds in $L^{\infty}(\Omega)$ thanks to a result in \cite{Rothe}.\\

Note that we can apply the same techniques for the generalized system:
\begin{equation}
\label{eq:2dgen}
\left\{
\begin{array}{rcl}
 u_t&=&\Delta u+f(u,v)\\
 v_t&=&-\delta v +g(x,u)
\end{array}
\right.
\end{equation}
if the following conditions hold
\begin{equation}
 \label{ineq:Condf1}
uf(u,v)\leq -\delta_1 u^p+\delta_2 uv+\delta_3
\end{equation}
and,
\begin{equation}
 \label{ineq:Condf2}
 |f(u,v)|\leq \delta_1 |u|^{p-1}+\delta_2|v|+\delta_3
\end{equation}
\begin{equation}
 \label{ineq:Condg}
 |\frac{\partial g}{\partial x}(x,u)|\leq c|u| \mbox{ and }  |\frac{\partial g}{\partial u}(x,u)|\leq K,
\end{equation}
where $p>2$, $\delta_i>0, i \in \{1,...,3\}$.

\textbf{Paper Organization}
After the present introduction, the theoretical results are presented in section $2$ and $3$. In section $2$, we prove, under some assumptions, the existence of the global attractor for the network of FHN-type \eqref{eq:equv} in $(L^2(\Omega)^d)^n,$ for any network topology. This allows us to show, in section $3$,  theoretical results on the synchronization onset of system \eqref{eq:equv}  with linear coupling functions. Then, we apply  these results  to complete and ring networks.  In section $4$, we  present the numerical simulations for fully connected and  unidirectionally coupled ring network. Each node, in the graph, is represented by a system of FitzHugh-Nagumo reaction-diffusion equations. This gives an insight on the relation between the number of neurons and the minimal coupling strength needed to reach the synchronization, with a particular attention on the effects of spatial dimension. For fully connected networks, our numerical simulations show that the minimal strength value for synchronization follows a ``$\frac{1}{n}$" law,  independently of the patterns induced by initial conditions. In unidirectional coupled ring network, the minimal strength value for synchronization follows a ``$n^2$'' law. Our conclusion is left to the last section.
\section{Existence of the global attractor}
\label{section:exofGA}
Now we prove the existence of the global attractor for the  dynamical system \eqref{eq:equv} in $\cH=(L^2(\Omega)^{d})^{n}$.
The global attractor is a compact invariant set for the flow that attracts all trajectories (see for example \cite{Ambrosio4, Marion, Robinson, Temam}). The existence of the global attractor is essential since it is a set  where the solutions asymptotically evolve.  In particular, all the patterns that we will see later  in our numerical simulations, belong, for enough  large time, to the global attractor. Also, our proof of synchronization given in section \ref{SecSynchro} uses $L^\infty$-bounds that we  prove in this section. Now, we specify some assumptions that we will assume throughout the article.
First, we assume that for all $i\in \{1,...,n\}$, and
for all $j\in \{1,...,s\}$,
 \begin{equation}\label{Ineq:F}
 u_i^jF^j(u_i,v_i) \leq -\delta_1| u_i^j|^p  + \delta_2|u_i^j|\sum_{k=1}^s|u_i^k |^{p_1}
 +\delta_3 |u_i^j|\sum_{k = 1}^{d-s} | v_i^k |   + \delta_4,
  \end{equation}
 with $p>2$,  $\delta_1, \delta_2, \delta_3>0$,  $0\leq p_1<p-1$, and,
\begin{equation}
\label{Ineq:F2}
|F^j(u_i,v_i)| \leq \delta_1| u_i^j|^{p-1}  + \delta_2\sum_{k=1}^s|u_i^k |^{p_1} + \delta_3\sum_{k = 1}^{d-s} | v_i^k |   + \delta_4.
\end{equation}
Condition \eqref{Ineq:F} generalizes \eqref{ineq:Condf1}. It indicates a decrease of order  $p$ at infinity and permit to obtain bounds in $L^q$ spaces. Condition \eqref{Ineq:F2} generalizes \eqref{ineq:Condf2} and allows us to apply Young inequalities in order to obtain bounds in $H^1$.
In our case, a typical example for which \eqref{Ineq:F}-\eqref{Ineq:F2} hold, is given by a function $F$ where
the  component $j$ reads as:
\begin{equation*}
\begin{array}{rcl}
F^j(u_i^1,...,u_i^s,v_i^1,...,v_i^{d-s})=&\displaystyle -a_{p-1}(u_i^j)^{p-1}+\sum_{k=0}^{p-2}\sum_{\alpha_{k1}+...+\alpha_{ks}=k}a_{k_1...k_s}\prod_{l=1}^s(u_i^l)^{\alpha_{kl}}\\
&\displaystyle+\sum_{l=1}^{d-s}b_kv_i^l,
\end{array}
\end{equation*}
with $a_{p-1}>0$, $a_{k_1...k_s}$, $b_k \in \R$ and $p$ even.
This simply means that, relatively to $u_i$, $F^j$ is polynomial of several variables, with the dominant term given by $(u_i^j)^{p-1}$, with negative coefficient,  and $p$ even.  The other terms have a degree lower than $p-1$.  Whereas $F^j$  is a linear function of $v_i$.

Moreover, in order to maintain the  effect of the decrease condition \eqref{Ineq:F},  we suppose that the coupling functions have a polynomial increase lower than $p-1$. This reads as:
\begin{equation}\label{Ineq:Hi}
| {H_i^j(u_1,...,u_n)} | \leq {\delta _4}(1+\sum_{k = 1}^n | u_k^j|^{p_1}),\,\,\,\,0<{p_1} < p - 1.
\end{equation}

Finally, we suppose that for all  $j \in \{1,...,d-s\}$,
\begin{equation}\label{ineq:Phix}
\left|  \frac{\partial \Phi^j }{\partial x_k}(x,u_i) \right| \leq \delta _5(1+\sum_{j=1}^s | u_i^j |),\,\,\,k\in \{1,...,N\},
\end{equation}
and,
\begin{equation}\label{ineq:Phiu}
\left| \frac{\partial \Phi^j }{\partial u^k_i}(x,u_i) \right| \leq {\delta _5}.
\end{equation}
Conditions \eqref{ineq:Phix} and \eqref{ineq:Phiu} generalize condition \eqref{ineq:Condg}.
They are not very restrictive and include functions $\Phi$ with spatial heterogeneity, that allow rich behavior, bifurcations and pattern formation (see \cite{Ambrosio0, Ambrosio4}).
We deduce from \eqref{ineq:Phix} and \eqref{ineq:Phiu} that for all $j \in \{1,...,d-s\}$,
\begin{equation}\label{ineq:Phi}
\left| \Phi^j(x,u_i) \right| \leq \delta_6(1+\sum_{l=1}^s\left| u_i^l\right|).
\end{equation}
Let us remark that all these assumptions appear naturally in the proof of the existence of the attractor of system \eqref{eq:equv} . They appear also in \cite{Marion, Robinson}.

 The following theorem gives the existence of the global attractor.

\begin{theorem}
\label{ThPrincipal}
Under assumptions \eqref{Ineq:F}-\eqref{ineq:Phiu}, the semi-group associated with \eqref{eq:equv} possesses a connected global attractor $\mathcal{A}$ in $\cH=(L^2(\Omega))^{nd}$. Furthermore, $\mathcal{A}$ is bounded in $(L^\infty(\Omega))^{nd}$.
\end{theorem}

The proof of theorem \ref{ThPrincipal} relies on a general result that gives the existence of the global attractor in Banach Spaces, see \cite{Temam}.  If there exists a bounded absorbing set $\mathcal{B}$ in $\cH$,  which means that $\mathcal{B}$ verifies the following condition:
\begin{equation}
\label{BorneAbsorbant}
\mbox{for all bounded set } B \subset \cH, \exists t_B; \forall t>t_B, S(t)B \subset \mathcal{B},
\end{equation}
 and if,

\begin{equation}
\label{S1RelaComp}
\mbox{for all bounded set } B \subset \cH, \exists t_B; \cup_{t\geq t_B}S(t)B \mbox{ is relatively compact in } \cH,
\end{equation}
%and if:
%\begin{equation}
%\label{S2DecrExpon}
%\mbox{for all bounded set } B \subset \cH,
%\sup_{\varphi \in B}|S_2(t)\varphi|_{H} \rightarrow 0 \mbox{ when } t \rightarrow +\infty,
%\end{equation}
%
then the $\omega$-limit set of $\mathcal{B}$, is an invariant connected compact set that attracts all the trajectories.
Therefore, we divide the proof of theorem \ref{ThPrincipal} into four parts, which for the reader's convenience,
we present as different lemmas. Before going into details let us briefly present the sketch of the proof.
We first show, in lemma \ref{Lem-absorbsetinH}, the existence of a bounded absorbing set in $\cH$, that is \eqref{BorneAbsorbant}.
Then, in lemma \ref{Lem-CompTraj}, we prove a result of compacity for trajectories, that is \eqref{S1RelaComp}%-\eqref{S2DecrExpon}
. More precisely, we establish the existence of a bounded absorbing set in $(H^1(\Omega))^{nd}$.
The result follows from the compact injection of $H^1(\Omega)$ in $L^2(\Omega)$.
Note that proving lemmas \ref{Lem-absorbsetinH} and \ref{Lem-CompTraj} gives the global attractor existence.
After, in lemma \ref{Lemma-Lq-Bounds}, we obtain the $(L^\infty(\Omega))^{nd}$-bound:
we show $(L^q(\Omega))^{nd}$-bounds for all $q \in \N$. Finally,  we prove theorem \ref{ThPrincipal},
by using a result that links $L^{\infty}$ and $L^q$-norms for linear parabolic equations, see \cite{Rothe}.\\
Let us introduce the following notations:
\begin{equation}
\label{NormLp}
|u|_{p,\Omega}=\big(\int_\Omega |u|^pdx \big)^{\frac{1}{p}},
\end{equation}
if $u$ is a real or vector valued function.
We also use
\begin{equation*}
||u||
\end{equation*}
to denote the euclidian norm for a real vector.
 The lemma below establishes the existence of a bounded absorbing set in  $\cH$.
\begin{lemma}
\label{Lem-absorbsetinH}
There exists an absorbing bounded set in $\cH$, that is, there is a constant $K$, such that for all initial conditions in $ \cH$ :
\begin{equation*}
(|u|_{2,\Omega}^2+|v|_{2,\Omega}^2)(t)\leq K \mbox{ for } t \mbox{ large enough}.
\end{equation*}
\end{lemma}
\begin{proof}
The proof mainly relies on the presence of $-\delta_1 |u_i^j|^p$  in \eqref{Ineq:F} and the  condition on $\sigma(x)$. In the following, parameters $\delta, c$ and $K$ are generic constants.
We multiply each scalar equation in \eqref{eq:equv}, for all  $i$ and $j$, by $u_i^j$ or $v_i^j$, and sum. We show, by using  Young inequality, $ab\leq \frac{a^p}{\epsilon^pp}+\frac{\epsilon^qb^q}{q}$, and as it has been done for equation \eqref{eq:2d}  that:
\begin{equation}
\label{ineq:Abs}
\frac{d}{dt}(|u|_{2,\Omega}^2+|v|_{2,\Omega}^2)+\delta_1(|u|_{2,\Omega}^2+|v|_{2,\Omega}^2)+\delta_2|u|_{p,\Omega}^p+\delta_3|\nabla u|_{2,\Omega}^2\leq \delta_4.
\end{equation}
Then by Gronwall inequality we have for $t$ large enough:
\begin{equation}
\label{ineq:Abs3}
(|u|_{2,\Omega}^2+|v|_{2,\Omega}^2)(t)\leq K.
\end{equation}
\end{proof}

Now, we will show the compactness of trajectories in $\cH$, by establishing \eqref{S1RelaComp}.
We have:
\begin{lemma}
\label{Lem-CompTraj}
There exists an absorbing bounded set in $(H_1(\Omega))^{nq}$, that is, there is a constant $K$, such that for all initial conditions in $\cH$  and $t$ large enough,
\begin{equation*}
|\nabla u(t)|_{2,\Omega}\leq K.
\end{equation*}
\end{lemma}
\begin{proof}
By integrating \eqref{ineq:Abs} between $t$ and $t+r$, we have for $t$ large enough and $\forall r>0$:
\begin{equation}
\label{Ineqgrauttr}
\int_t^{t+r}|\nabla u|_{2,\Omega}^2+\delta\int_t^{t+r}|u|_{p,\Omega}^p\leq K+\delta r.
\end{equation}
Now, we multiply each component of the first equation of \eqref{eq:equv} by $-\Delta u_i^j$, we integrate and sum over $i$ and $j$. We obtain:
\begin{equation}
\frac{1}{2}\frac{d}{dt}|\nabla u|_{2,\Omega}^2=-\sum_{i=1}^n\sum_{j=1}^s\left( \int_\Omega (F^j(u_i,v_i)\Delta u_i^j+q_j\Delta u_i^j\Delta u_i^j+H_i^j(u_1,...,u_n)\Delta u_i^j\right),
\end{equation}
which, thanks to \eqref{Ineq:F2} and \eqref{Ineq:Hi} leads to:

\begin{align*}
\frac{1}{2}\frac{d}{dt}|\nabla u|_{2,\Omega}^2+q|\Delta u|_{2,\Omega}^2\leq & c\sum_{i=1}^n\sum_{j=1}^s \int_{\Omega}\Big{(}1+\sum_{l=1}^s|u_i^l|^{p-1}+\sum_{k=1}^{d-s}|v_i^k|+\sum_{l=1}^n|u_l^j|^{p_1}\Big{)}|\Delta u_i^j|\\
&\leq c\sum_{i=1}^n\sum_{j=1}^s\int_\Omega\Big(\frac{c}{2q}\big{(}1+\sum_{l=1}^s|u_i^l|^{p-1}+\sum_{k=1}^{d-s}|v_i^k|+\sum_{l=1}^n|u_l^j|^{p_1}\big{)}^2+\frac{q}{2c}(\Delta u_i^j)^2\Big),
\end{align*}
where $q=\min_{i\in\{1,...,s\}}q_i$ and $c$ is a generic constant. It follows that:
\begin{align*}
\frac{1}{2}\frac{d}{dt}|\nabla u|_{2,\Omega}^2
&\leq c\sum_{i=1}^n\sum_{j=1}^s\int_\Omega\big{(}1+\sum_{l=1}^s|u_i^l|^{2p-2}+\sum_{k=1}^{d-s}|v_i^k|^2+\sum_{l=1}^n|u_l^j|^{2p_1}\big{)}\\
&\leq c(1+|u|_{2p-2,\Omega}^{2p-2}+|v|_{2,\Omega}^2).
\end{align*}
%\section*{Remerciements}
% Remerciements - texte ici
Thanks to the techniques we use in lemma \ref{Lemma-Lq-Bounds}, we can prove that for $t$ large enough:
\begin{equation}
|u|_{2p-2,\Omega}^{2p-2}\leq K.
\end{equation}
Then we can apply the uniform Gronwall lemma (see appendix), and show that:
\begin{equation*}
|\nabla u|_{2,\Omega}^2(t)\leq K \mbox{ for t  large enough}.
\end{equation*}
It remains to find a bound for $|\nabla v|_{2,\Omega}$.
For all $i\in \{1,...,n\}$, and  $k\in \{1,...,N\}$, we have:
\begin{equation}
\label{eq:wi}
\frac{d}{2dt}|v_{ix_k}|^2_{2,\Omega}=\int_\Omega(-\sigma'_{x_k}(x)v_{i}\cdot v_{ix_k}-\sigma(x) v_{ix_k}\cdot v_{ix_k} +\Phi'_{x_k}(x,u_i)\cdot v_{ix_k}+\sum_{j=1}^su_{ix_k}^j\Phi'_{u_i^j}(x,u_i)\cdot v_{ix_k}),
\end{equation}
with $v_{ix_k}=\frac{\partial v_i}{\partial x_k} $.
Using Young and Cauchy-Schwarz inequalities, we find, thanks to \eqref{ineq:Phix} and  \eqref{ineq:Phiu},  for enough large time:
\begin{equation}
\label{eq:wi2}
\frac{d}{dt}|v_{ix_k}|^2_{2,\Omega}\leq -\frac{\sigma}{2}|v_{ix_k}|^2_{2,\Omega}+K.
\end{equation}
Finally, the result follows by using Gronwall inequality and summing over $i$.
\end{proof}

\begin{lemma}
\label{Lemma-Lq-Bounds}
For all $q\in \N$, there exists an absorbing bounded set in $(L^q(\Omega))^{nd}$, that is, for all $q\in \N$ there is a constant $K_q$ such that for all initial conditions in $ \cH$ :
\begin{equation*}
(|u|_{q,\Omega}^q+|v|_{q,\Omega}^q)(t)\leq K_q \mbox{ for } t \mbox{ large enough}.
\end{equation*}
\end{lemma}
\begin{proof}
In the following, $K$ is a generic constant.
We multiply the first equation of \eqref{eq:equv} by $|u^j_i|^{2k-2}u_i^j$, for all $i,j$ and integrate. We obtain:
\begin{equation}
\label{eq:uijalk}
\frac{1}{2k}\frac{d}{dt}\int_\Omega |u_i^j|^{2k}=\int_\Omega F(u_i,v_i)|u^j_i|^{2k-2}u_i^j +\int_\Omega q^j\Delta u_i^j|u^j_i|^{2k-2}u_i^j + H_i|u^j_i|^{2k-2}u_i^j.
\end{equation}
But we have:
\begin{equation}
\label{a}
 \int_\Omega q^j\Delta u_i^j|u^j_i|^{2k-2}u_i^j=-(2k-1)\int_\Omega|\nabla u_i^j|^2|u^j_i|^{2k-2}.
\end{equation}
Thanks to \eqref{Ineq:F} and to \eqref{Ineq:Hi}, it follows that:
\begin{equation}
\frac{1}{2k}\frac{d}{dt}\int_\Omega |u_i^j|^{2k}\leq -\delta_1{'}\int_\Omega |u^j_i|^{2k-2+p}  +\delta_2{'}\sum_{k=1}^{d-s}\int_\Omega|u_i^j||v_i^k| + \sum_{k=1}^n\int_\Omega |u_i^k|^{p_1}|u^j_i|^{2k-2}u_i^j.
\end{equation}
Hence by Young inequality,
\begin{equation}
\frac{1}{2k}\frac{d}{dt} |u|^{2k}_{2k,\Omega}\leq  -\delta_1 |u|_{2k,\Omega}^{2k}+K.
\end{equation}
By Gronwall lemma:
\begin{equation}
 |u|^{2k}_{2k,\Omega}\leq K,
\end{equation}
for $t$ large enough.

Besides, by similar techniques, we have:
\begin{equation}
\frac{d}{dt}|v_i|_{2k,\Omega}^{2k}=-\sigma|v_i^j|_{2k,\Omega}^{2k}+K.
\end{equation}
Then, the result follows by Gronwall Lemma.
\end{proof}

\begin{proof}[Proof of theorem \ref{ThPrincipal}]
It remains to prove the $L^\infty$-bound .
For all $i\in \{1,...,n\}$, and for all $j\in \{1,...,s\}$, we have:
\begin{equation}
u_i^j(t)=\mathcal{T}(t)u^j_{i0}+\int_0^t\mathcal{T}(t-\tau)\{F^j(u_i^j,v_i^j)(\tau)+H_i^j(u_1(\tau),...,u_n(\tau))+u_i^j(\tau)\}
\end{equation}
where $\mathcal{T}$ represents the semi-group associated with $\frac{\partial \varphi}{\partial t}-q^j\Delta \varphi + \varphi=0$ and  Neumann boundary conditions.
We know, see \cite{Rothe}, lemma 3 p 25, that  $\mathcal{T}$ verifies:
\begin{equation}
|\mathcal{T}(t)\varphi|_{\infty,\Omega}\leq cm(t)^{-\frac{1}{2}}e^{-\lambda t}|\varphi|_{2N,\Omega},
\end{equation}
where $m(t)=\min(1,t)$, $\lambda$ is the smallest eigenvalue of the operator  $I-q^j\Delta$, and $c$ is a positive constant.
This allows us to conclude.
\end{proof}
\section{Identical synchronization  of a  network of $n$ reaction-diffusion systems}
\label{SecSynchro}
Now, as we have shown the existence of the global attractor, we can think about synchronization of solutions within this attractor. Therefore, in the present section, we focus on this ubiquitous  phenomenon  and restrict ourselves to identical synchronization, in  networks with linear  coupling. We will  determine sufficient coupling strength values ensuring the identical synchronization in complex networks \eqref{eq:equv}. We refer to  \cite{A5, C1, C1p,Dorfler1, Dorfler2, Pikovsky} for relevant works on the synchronization phenomenon. These values  reveal the dependence on the number of nodes, the coupling configuration, the properties of each subsystem. Such a result holds in case  of spatial heterogeneity. This point will be discussed in more details in the last section. Technically, to establish the synchronization, we exhibit a Lyapunov function for the network. Since we obtained $L^\infty$-bounds in the previous section, it is always possible to find such a Lyapunov function provided that the coupling strength  is large enough. We rely on previous works  for networks of ODE's (see \cite{B4, B5}) and assume that the connectivity matrix has zero row and column sums. However, we extend the results to networks of reaction-diffusion systems and improve computations, which lead us to find the same threshold synchronization values, see below and theorems \ref{CSSync} and \ref{th:paths}.
\begin{definition}[see \cite{Ambrosio1}]\label{t4}
Let $U(t)=(U_1(t), U_2(t),..., U_n(t))$ be a network with a given topology. We say that $U$ synchronizes identically if,
\[\mathop {\lim }\limits_{t \to  + \infty } \sum_{i = 1}^{n - 1} | U_i(t) - U_{i + 1}(t) |_{2,\Omega }   = 0.\]
where notations are those of \eqref{eq:equv} and \eqref{NormLp}.
\end{definition}
%\subsection{Synchronization in linearly coupled networks}
We consider the network,
\begin{equation}
\label{eq:equvLC}
\left\{ \begin{array}{rcl}
  u_{it} &=& F({u_i,v_i}) + Q\Delta {u_i} + \sum_{k=1}^nc_{ik}u_k,\quad i\in \{1,...,n\} \\
 v_{it} &=&  - \sigma (x)v_i + \Phi (x,u_i),
 \end{array} \right.
 \end{equation}
 with Neumann Boundary conditions on $\partial \Omega$
which is a particular case of \eqref{eq:equv} with $U_i=(u_i,v_i)$, where :\[H_i(u_1,....,u_n)=\sum_{k=1}^nc_{ik}u_k.\]
We will assume that the matrix $G=(c_{ik})_{ 1\leq i,k \leq n}$  has vanishing row and column-sums and non-negative off-diagonal elements, i.e., $c_{ik}\geq 0$ for $i\neq k$, and
\begin{equation}
\label{eq:VanRowSum}
c _{ii} =  - \sum_{k = 1,k \ne i}^n c _{ik}=- \sum_{k = 1,k \ne i}^n c_{ki}.
\end{equation}
The  connectivity matrix $G$ defines the graph topology as well as the coupling strength between nodes. Indeed, there is an edge between node $i$ and node $k$ if and only if  $c_{ik}>0$.  We consider here an arbitrary connected and directed network of  linearly coupled RD systems satisfying \eqref{eq:VanRowSum}. Obviously, symmetrically networks with vanishing row-sums are a particular case of the one considered here.
 In order to ensure the existence of  the global attractor, we assume that the network \eqref{eq:equvLC} verifies the assumptions \eqref{Ineq:F}-\eqref{ineq:Phiu}. Note that \eqref{Ineq:Hi} holds automatically thanks to the linear coupling.
 Before going more into details, let us summarize the main ideas of the section.  We establish the synchronization  result in theorem \ref{CSSync}  which relies on the statement of a Lyapunov function $V$  that reads as the sum of the norms of all  vectors $U_i-U_j$. More precisely,  the first step exhibits a diagonal definite positive  matrix $A$ that annihilates the effects of nonlinear terms of \eqref{eq:equvLC}, this is done in lemma  \ref{lemma:matA}, then the proof follows from computations that use property \eqref{eq:VanRowSum}, which itself uses all the couples $(i,j)$ in $V$. However, the result of theorem \ref{CSSync} can not be applied directly to concrete networks. This is the purpose of  theorem \ref{th:paths}, which gives a sufficient condition for synchronization applicable in networks with arbitrary topology. Its proof relies on the computation of the sum of all minimal lengths joining any couple of nodes passing through a given edge of the graph. This idea comes from  lemma \ref{lem-MajNormeWij}  which  bounds all terms in the left hand side of \eqref{eq:CSSync}  (in which all the couples $(i,j)$ appear) by terms of the right hand side (only the couples $(i,j)$ corresponding to non zero coupling, i.e. $c_{ij}\neq 0$, i.e. edges of the graph,  appear).

We start with the lemma \ref{lemma:matA} that establishes the existence of a diagonal definite-positive matrix $A$ that annihilates  the effects of nonlinear terms. Let us denote by $X$ an arbitrary vector of $\R^d$, $X_s$ the vector of the $s$ first coordinates and $X_{d-s}$, the vector of the $d-s$ last coordinates of $X$.

\begin{lemma}
\label{lemma:matA}
For each solution $(U_1,U_2,...,U_n)$ of \eqref{eq:equvLC},  there exists a time $T$, a positive constant $\kappa$ and a definite positive diagonal  matrix $A \in \R^{s}\times \R^s$, such that for all $t>T$, for all $X\in \R^d$, and for all $i,j \in \{1,...,n\}$,
\begin{equation*}
\begin{array}{rl}
  \displaystyle   X_s.\left[ {\int_0^1 {DF(\theta {U_j} + (1 - \theta ){U_i})d\theta } } \right]X+&\\
   X_{d-s}.\left(- \sigma (x)X_{d-s} + \left[ {\int_0^1 {D\Phi(\theta {u_j} + (1 - \theta ){u_i})d\theta } } \right]X_{s}  \right)-AX_{s}.X_{s}&\leq -\kappa||X_{d-s}||^2.
 \end{array}
\end{equation*}
\end{lemma}
\begin{proof}
This follows from the $L^\infty$-bound of solutions of \eqref{eq:equv}, and by using young inequality.
\end{proof}

We will now present the main result that furnishes sufficient conditions for identical synchronization. To this aim, we introduce the following notations:
for all $i,j\in\{1,...,n\}$, $w_{ij}=u_j-u_i$, $z_{ij}=v_j-v_i, i,j\in \{1,...,n\}$,  $\epsilon_{ij}=\frac{c_{ij}+c_{ji}}{2} $ and $a=\max_{l\in \{1,...,s\} } A_{ll}$.
\begin{theorem}\label{CSSync}
If we assume that
\begin{equation}
\label{eq:CSSync}
\displaystyle \frac{a}{n}\sum_{i<j} | w_{ij}|_{2,\Omega}^2< \sum_{i<j}^n\epsilon_{ij}|w_{ij}|_{2,\Omega}^2
\end{equation}
 then the system \eqref{eq:equvLC} synchronizes in the sense of definition \ref{t4}.
\end{theorem}
\begin{proof}
We split $G$ into the sum of two symmetric and antisymmetric matrices, $E=(\epsilon_{ik}), i,k\in\{1,...,n\}$ and $L=(\delta_{ik}), i,k\in\{1,...,n\}$:
\begin{equation}\label{eq:sum-E-L}
G=E+L,
\end{equation}
where,

\begin{equation}\label{eq:del-ik}
\delta _{ik}=
   \displaystyle\frac{1}{2}({c_{ik}} - {c_{ki}}).\,\,\,\,\,\,\
\end{equation}

One can easily check that both matrices $E$ and $L$ have zero row sums.
As $w_{ij}=u_j-u_i$, $z_{ij}=v_j-v_i, i,j\in \{1,...,n\}$, we obtain,
\begin{equation}
\label{eq:equvLCNot}
\left\{ \begin{array}{rcl}
  w_{ijt} \!\!\!&=&\!\!\! F({u_j,v_j})-F({u_i,v_i}) + Q\Delta w_{ij} + \sum_{k=1}^n(\epsilon_{jk}w_{jk}-\epsilon_{ik}w_{ik} +\delta_{jk}w_{jk}-\delta_{ik}w_{ik}), \\
 z_{ijt} \!\!\!&=&\!\!\!  - \sigma (x)z_{ij} + \Phi (x,u_j)-\Phi (x,u_i).
 \end{array} \right.
 \end{equation}
 Besides,
\[F({u_j,v_j}) - F({u_i,v_i}) = \int_0^1 {\frac{d}{{d\theta }}F(\theta {U_j} + (1 - \theta ){U_i})d\theta }  = \left[ {\int_0^1 {DF(\theta {U_j} + (1 - \theta ){U_i})d\theta } } \right](U_j-U_i),\]
where $DF$ is the $s\times d$ Jacobian matrix of $F$,
and,
\[\Phi (x,u_j)-\Phi (x,u_i)= \int_0^1 {\frac{d}{{d\theta }}\Phi(\theta {u_j} + (1 - \theta ){u_i})d\theta }  = \left[ {\int_0^1 {D\Phi(\theta {u_j} + (1 - \theta ){u_i})d\theta } } \right]w_{ij}.\]
Hence, we can write,
\begin{equation}
\label{eq:equvLCNot-2}
\left\{ \begin{array}{rcl}
  w_{ijt} &=&  \left[ \displaystyle\int_0^1 {DF(\theta {U_j} + (1 - \theta ){U_i})d\theta } \right](U_j-U_i)+ Q\Delta w_{ij} +\\
   & &\displaystyle \sum_{k=1}^n(\epsilon_{jk}w_{jk}-\epsilon_{ik}w_{ik}+\delta_{jk}w_{jk}-\delta_{ik}w_{ik}), \\
 z_{ijt} &=&  - \sigma (x)z_{ij} + \left[ \displaystyle\int_0^1 {D\Phi(\theta {u_j} + (1 - \theta ){u_i})d\theta } \right]w_{ij}.
 \end{array} \right.
 \end{equation}
where $i,j\in \{1,...,n\}$.
Now, let us introduce the following function:
\begin{equation}
\label{eq:V}
V(t) = \sum_{i = 1}^n \sum_{j = 1}^n \left(|w_{ij}|^2_{2,\Omega}+|z_{ij}|^2_{2,\Omega}\right).
\end{equation}
In order to reach the synchronization, it is sufficient to find conditions ensuring that $V$ is a Lyapunov function with negative orbital derivative. As the graph is connected, it would be natural to include in $V$, only terms $w_{ij}$ corresponding to non-zero coefficients $\epsilon_{ij}$. However, as we will see, including all terms $w_{ij}$ in $V$ will allows to vanish sums.

Then,

\begin{eqnarray*}
 \displaystyle \frac{1}{2}\frac{d}{dt}V &=&  \displaystyle \sum_{i = 1}^n \sum_{j = 1}^n \int_\Omega   \displaystyle \left( w_{ij}.w_{ijt}+z_{ij}.z_{ijt} \right)\\\displaystyle
  &=&  \displaystyle \sum_{i = 1}^n \sum_{j = 1}^n \large{(}\int_\Omega w_{ij}.(\left[ {\int_0^1 {DF(\theta {U_j} + (1 - \theta ){U_i})d\theta } } \right]
  (U_j-U_i)+ Q\Delta w_{ij}\\ \displaystyle
 & &+\displaystyle \sum_{k=1}^n(\epsilon_{jk}w_{jk}-\epsilon_{ik}w_{ik} +\delta_{jk}w_{jk}-\delta_{ik}w_{ik}) ) \\ \displaystyle
& &+\displaystyle \int_\Omega  z_{ij}.(- \sigma (x)z_{ij} + \left[ {\int_0^1 {D\Phi(\theta {u_j} + (1 - \theta ){u_i})d\theta } } \right] w_{ij}  ) \large{)}.
\end{eqnarray*}

Now by lemma \ref{lemma:matA}, and using green formula, there exists a  definite positive  diagonal matrix $A$ such that,

\begin{equation*}
\begin{array}{rl}
  \frac{1}{2}\frac{dV}{dt}\leq&\displaystyle \sum_{i = 1}^n \sum_{j = 1}^n \left(\int_\Omega Aw_{ij}.w_{ij}+\int_\Omega w_{ij}.
  \left( \sum_{k=1}^n(\epsilon_{jk}w_{jk}-\epsilon_{ik}w_{ik})\right)\right. \\
  &\left.+\displaystyle\int_\Omega w_{ij}.\left( \sum_{k=1}^n(\delta_{jk}w_{jk}-\delta_{ik}w_{ik})\right)-\kappa|z_{ij}|^2_{2,\Omega}\right).\\
 \end{array}
\end{equation*}
The last term in the above equation vanishes, indeed:
\begin{eqnarray*}
\sum_{i,j,k = 1}^n w_{ij}.\big{(}\delta_{jk}w_{jk}-\delta_{ik}w_{ik}\big{)}=&2\sum_{i,j,k = 1}^n w_{ij}.\delta_{jk}w_{ik}\mbox{ (because } \sum_{k=1}^n\delta_{jk}=0) \\
=&2\sum_{i,j}^n \delta_{jj}||w_{ij}||^2+2\sum_{i,j=1}^n\sum_{k\neq j} \delta_{jk}w_{ij}.w_{ik}.
\end{eqnarray*}
Obviously since $\delta_{jj}=0$,
\begin{eqnarray*}
\sum_{i,j}^n \delta_{jj}||w_{ij}||^2=&0.
\end{eqnarray*}
Moreover, since $\delta_{jk}=-\delta_{kj}$, we have, for all $i \in \{1,...,n\}$,
\[\sum_{j=1}^n\sum_{k\neq j} \delta_{jk}w_{ij}.w_{ik}=0.\]

Therefore, asymmetric connectivity matrices with zero column-sums and row-sums can be treated as symmetric matrices.
Now we deal with the other terms.
We have:
\begin{eqnarray*}
\displaystyle \sum_{i=1}^n\sum_{j=1}^nw_{ij}.\sum_{k=1}^n(\epsilon_{jk}w_{jk}-\epsilon_{ik}w_{ik})&=&\sum_{i,j,k=1}^n \epsilon_{jk}w_{ij}.w_{jk}-\sum_{i,j,k=1}^n \epsilon_{ik}w_{ij}.w_{ik}\\
&=&\sum_{i,j,k=1}^n \epsilon_{jk}w_{ij}.w_{jk}-\sum_{i,j,k=1}^n\epsilon_{jk} w_{ji}.w_{jk}\\
&=&2\sum_{i,j,k=1}^n\epsilon_{jk}w_{ij}.w_{jk}\\
&=&2\sum_{i,j,k=1}^n\epsilon_{jk}(w_{ik}.w_{jk}-w_{kj}^2),
\end{eqnarray*}
since $w_{ij}=w_{ik}+w_{kj}$.
Moreover,
\begin{eqnarray*}
\sum_{i,j,k=1}^n w_{ik}.\epsilon_{jk}w_{jk}&\leq &\sum_{i,j,k=1}^n\frac{1}{2}\epsilon_{jk}(||w_{ik}||^2+||w_{jk}||^2)\\
&=&\frac{1}{2}\sum_{i,k=1}^n||w_{ik}||^2\sum_{j=1}^n\epsilon_{jk}+\frac{n}{2}\sum_{j,k=1}^n\epsilon_{jk}||w_{jk}||^2\\
&=&\frac{n}{2}\sum_{j,k=1}^n\epsilon_{jk}||w_{jk}||^2,
\end{eqnarray*}
since $\displaystyle \sum_{j=1}^n\epsilon_{jk}=0$.
Finally, we obtain:
\begin{equation*}
\begin{array}{rcl}
  \frac{1}{2}\frac{dV}{dt} &\leq &  \displaystyle \sum_{i,j = 1}^n  (-\kappa|z_{ij}|^2_{2,\Omega}+\int_\Omega Aw_{ij}.w_{ij})-n\int_\Omega \big{(} \sum_{j,k=1}^n(\epsilon_{jk}||w_{jk}||^2\big{)}\\
  &=& 2\displaystyle \displaystyle \sum_{i<j}  (-\kappa|z_{ij}|^2_{2,\Omega}+\int_\Omega Aw_{ij}.w_{ij})-2n\int_\Omega \big{(} \sum_{i<j}^n(\epsilon_{ij}||w_{ij}||^2\big{)}\\
    &\leq& -2\kappa\displaystyle  \sum_{i<j}  |z_{ij}|^2_{2,\Omega}+2a\sum_{i<j} |w_{ij}|_{2,\Omega}^2-2n \sum_{i<j}^n\epsilon_{ij}|w_{ij}|_{2,\Omega}^2\\
    &<&0,
\end{array}
\end{equation*}
if $(w_{ij},z_{ij})\neq 0$ thanks to hypothesis \eqref{eq:CSSync}.
\end{proof}

 Now we would like to prove theorem \ref{th:paths} which gives a sufficient condition for synchronization which can be applied for general networks.  An apparent difficulty comes from the fact that in the left-hand side of \eqref{eq:CSSync}, all the variables appear while in the right-hand side, because of the parameters $\epsilon_{ij}$, only variables corresponding to edges appear. The following lemma allows to obtain a bound with only terms corresponding to edges.

\begin{lemma}
\label{lem-MajNormeWij}
For all $i,j \in \{1,...,n\}$, and for all sequence, $(i_l)_{l\in\{0,...,k\}}$,  with $i_0=i,...,i_k=j$, we have:
\begin{equation}
\label{eq:MajNormeWij}
||w_{ij}||^2\leq k(\sum_{l=0}^{k-1}||w_{i_li_{l+1}}||^2).
\end{equation}
\end{lemma}
\begin{proof}
We write:
\begin{eqnarray*}
||w_{ij}||^2&=&||\sum_{l=0}^{k-1} w_{i_li_{l+1}}||^2\\
&=&(\sum_{l=0}^{k-1} w_{i_li_{l+1}}).(\sum_{l=0}^{k-1} w_{i_li_{l+1}})\\
&=&\sum_{l=0}^{k-1}||w_{i_li_{l+1}}^2||+2\sum_{l<m}w_{i_li_{l+1}}w_{i_mi_{m+1}}\\
& \leq &k\sum_{l=0}^{k-1} ||w_{i_{l}i_{l+1}}||^2 \quad \mbox{         by Young inequality}.
\end{eqnarray*}
\end{proof}

Note that the lemma is valid for an arbitrary sequence, but we will use it for sequences corresponding to edges in the graph.

Without taking care of the  edges direction in the graph, for each $(i,j)$, $i<j$,
we  choose a unique  path of minimal length in the graph joining nodes $i$ and $j$.
We denote this path $P_{ij}$, $l(P_{ij})$ its length, and its nodes by: $(i_l)_{l\in\{0,...,length(P_{ij})\}}$,  with $i_0=i,...,i_{length(P_{ij})}=j$.
For each $(k,l)$, $k<l$ corresponding to an edge in the graph (i.e. $\epsilon_{kl}\neq 0$),
we define $\alpha_{kl}$ as the sum of all the lengths of  the minimal paths passing trough the edge $(k,l)$. If $(k,l)$ is not an edge of the graph, we set $\alpha_{kl}=0$:
\begin{equation}
\alpha_{kl}=\left\{
\begin{array}{l}
\sum_{i<j, (k,l)\in P_{ij}}length(P_{ij}) \mbox{ if } (k,l) \mbox{ is an edge of the graph, }\\
0\mbox{ if } (k,l) \mbox{ is not an edge of the graph.}
\end{array}
\right.
\end{equation}
Figure \ref{FigExpleCalculakl} gives an example of the computation of these coefficients.
\begin{theorem}\label{th:paths}
Let us assume that for each edge $(k,l)$, we have,
\[\frac{a}{n}\alpha_{kl}<\epsilon_{kl},\]
then system \eqref{eq:equvLC} synchronizes in the sense of definition \ref{t4}.
\end{theorem}
\begin{proof}
Now, the proof of theorem \ref{th:paths} follows from theorem \ref{CSSync} and lemma \ref{lem-MajNormeWij}.
We have,
\begin{eqnarray*}
\displaystyle \sum_{i<j}||w_{ij}||^2&\leq &\sum_{i<j}length(P_{ij})\sum_{l=0}^{length(P_{ij})-1}||w_{i_{l}i_{l+1}}||^2 \mbox{ thanks to  lemma  }\ref{lem-MajNormeWij}\\
 &=& \sum_{k<l}\alpha_{kl}||w_{kl}||^2 \mbox{ by reordering the terms of the sum along the edges of the graph.}
\end{eqnarray*}

Now the result follows from theorem \ref{CSSync}. Such a result has been found in \cite{B4, B5} for ODE systems.
\end{proof}

\begin{figure}[h]
\begin{center}

   \includegraphics[scale=0.15]{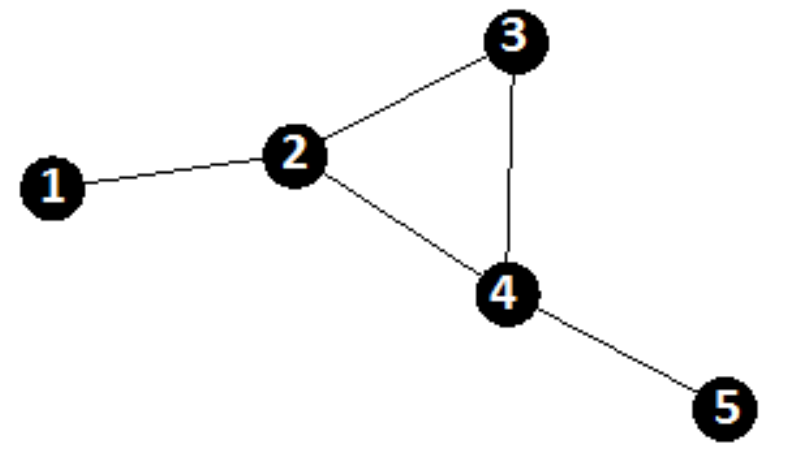}
   \begin{quotation}
   \caption{Example of calculus of the number $\alpha_{kl}$. The minimal path joining $1$ and $2$ is $P_{12}=$ $1-2$,
   $P_{13}=1-2-3$, $P_{14}=1-2-4$, $P_{15}=1-2-4-5$, $P_{23}=2-3$, $P_{24}=2-4$,$P_{25}=2-4-5$, $P_{34}=3-4$, $P_{35}=3-4-5$, $P_{45}=4-5$.
   It follows that $\alpha_{23}=2+1=3$.    }
   \end{quotation}
   \label{FigExpleCalculakl}
   \end{center}
\end{figure}

\begin{corollary}[Fully connected network]
We assume that $\forall i,j\in\{1,...,n\}$, $i\neq j$ $\epsilon_{ij}>\frac{a}{n}$, then  \eqref{eq:equvLC} synchronizes in the sense of definition \ref{t4}.
\end{corollary}
\begin{proof}
This result comes obviously from theorem \ref{th:paths}, since for fully connected network, $\alpha_{ij}=1$. Note that in the particular case of fully connected network, we could also conclude from theorem \ref{CSSync}.
\end{proof}
\begin{corollary}[Unidirectionally ring network]
We assume that the connectivity matrix $G=(c_{ij}),1\leq i,j\leq n,$ is given by  $c_{ii}=-c<0$  $ \forall i \in \{1,...,n\}$, $c_{ii+1}=c_{n1}=c$  $\forall i \in \{1,...,n-1\}$, and $c_{ij}=0$ otherwise. Then, if we assume that
\begin{equation*}
c>\left\{
\begin{array}{rcl}
\frac{a}{12}(n^2-1)& \mbox{ if }& n \mbox{ odd }\\
\frac{a}{12}(n^2+2)& \mbox{ if }& n \mbox{ and } \frac{n}{2} \mbox{ even }  \\
\frac{a}{12}(n^2+8)& \mbox{ if }& n \mbox{ even } \mbox{ and } \frac{n}{2}   \mbox{ odd }\\
\end{array}
\right.
\end{equation*}
 the network \eqref{eq:equvLC} synchronizes in the sense of definition \ref{t4}.
\end{corollary}
\begin{proof}
We start with the case of $n$ odd. For each couple of nodes of the graph, there is a unique path of minimal length joining the nodes.
If we suppose that $n=2k+1$, then for each node indexed by $l$:
\begin{equation*}
\begin{array}{rcl}
\alpha_{ll+1}&=&(1+...+k)+(2+....+k)+...+(k-1+k)+k\\
&=&\sum_{i=1}^k\sum_{j=i}^kj\\
&=&\frac{(n-1)(n+1)n}{24}.
\end{array}
\end{equation*}
The figure \ref{fig:ExpleCalculaklRing}-a gives an example of such a network.
In the case where of $n$ even, we assume that $n=2k$.  For each couple of nodes $(i,j)$ in the graph, if the distance between $i$ and $j$ is less than $k$,
there exists  a unique path of minimal length joining the nodes. But if the distance between $i$ and $j$ is equal to $k$, there exists two distinct paths of minimal length joining the nodes.
Therefore, we can choose for each couple $(i,j)$ of distance $k$, alternatively the minimal path trough the left and trough the right (i.e. for example, for node 1, the minimal path of length $k$ trough the left). Then we find, if $\frac{n}{2}$ is even,
\begin{equation*}
\begin{array}{rcl}
\alpha_{ll+1}&=&(1+...+k)+(2+....+k-1)+...+(k-1+k)+k-1\\
&=&\sum_{i=1}^k\sum_{j=i}^kj-\frac{k}{2}k\\
&=&\frac{n(n^2+2)}{24}.
\end{array}
\end{equation*}
Figure \ref{fig:ExpleCalculaklRing}-b gives an example of such a network. If $\frac{n}{2}$ is odd and in the worst case,
\begin{equation*}
\begin{array}{rcl}
\alpha_{ll+1}&=&(1+...+k)+(2+....+k-1)+(3+....+k)+...+(k-1)+k\\
&=&\sum_{i=1}^k\sum_{j=i}^kj-\frac{k-1}{2}k\\
&=&\frac{n(n^2+8)}{24}.
\end{array}
\end{equation*}
\end{proof}

\begin{figure}[h]
\begin{center}
\includegraphics[scale=0.21]{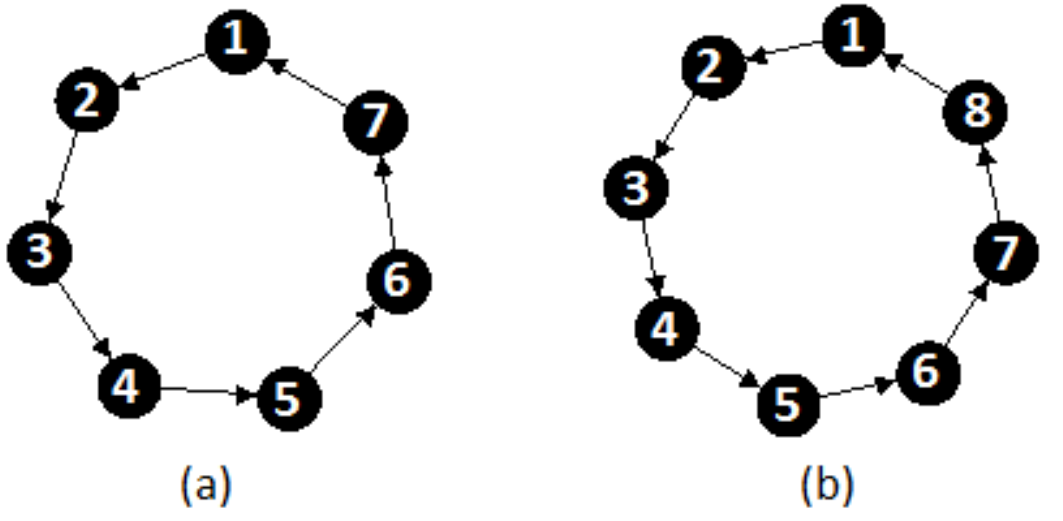}

\end{center}
\caption{Unidirectionally ring. In panel a), the graph has an odd number of nodes, $n=7$. There exists a unique minimal path  joining each couple of nodes in the graph. For example, the computation of $\alpha_{12}$ is given
by $\alpha_{12}=(1+2+3)+(2+3)+(3)$. The order of the computation follows from the counting of all the lengths of the minimal paths passing trough the edge $(1,2)$, starting at $1$,$7$ and
 $6$. In panel b), the graph has an even number of nodes, $n=8$. If the distance between two nodes is equal to $\frac{n}{2}=4$, there exist two distinct minimal paths joining these
 nodes. For example, we can link node $1$ and node $5$ either by the path $1-2-3-4-5$ or by $1-8-7-6-5$. Therefore, we choose the path $1-2-3-4-5$ to connect  nodes $1$ and $5$, whereas we choose the path
 $2-1-8-7-6$ to connect $2$ and $6$, and so on. Then computation of $\alpha_{12}$ is given
by $\alpha_{12}=(1+2+3+4)+(2+3)+(3+4)$.}
\label{fig:ExpleCalculaklRing}
\end{figure}

%\begin{figure}[H]
%\begin{center}
%   \includegraphics[scale=.9]{Figures/6new1}
%   \begin{quotation}
%   \caption{\label{fig2} The figure $(a)$ represents  a unidirectionally coupled ring network; $(b)$ represents its symmetrized analogs with half the coupling strength per link. Arrows indicate direction of coupling along an edge, and edges without arrows are coupled bidirectionally.}
%   \end{quotation}
%   \end{center}
%\end{figure}

\section{Numerical simulations}
In this section, we consider networks of type \eqref{eq:equvLC}   with:
\begin{align*}
d=2, s=1,\\
F(u,v)=\frac{1}{\epsilon}(-u^3+3u-v), Q=\frac{d_u}{\epsilon}\\
\sigma(x)=b,  \quad \phi(x,u)=au+c,
\end{align*}
i.e. each node is represented by the following  reaction-diffusion system of FitzHugh-Nagumo type,
\begin{equation}\label{61}
\left\{ \begin{array}{l}
 \epsilon {u_t} = d_u\Delta u-u^3+3u - v \,\,\,\,\,\,\text{on}\,\,\,\,\Omega  \times \textbf{R}^+ \\
 {v_t} = au - bv + c(x) \,\,\,\,\,\,\,\,\,\,\,\,\,\,\,\,\,\,\,\,\,\,\,\text{on}\,\,\,\,\Omega  \times {\textbf{R}^ + } \\
 \end{array} \right.
 \end{equation}
 where $u=u(x,t), v=v(x,t)$, $d_u, a, b>0$, $\Omega \subset \textbf{R}^N$ is a regular  bounded  open set, and with Neumann zero flux conditions on the boundary.

We perform numerical simulations in two cases, fully connected networks in one hand and unidirectionally coupled ring networks in the other hand, both with linear coupling. As far as we know, only few results exists for networks of reaction-diffusion systems, the case of a chain network has been partially investigated in \cite{Ambrosio1}.   We use the following parameters values:
\begin{equation}
\label{eq:parameter-values}
a=1, b=0.001,  \epsilon=0.1, d_u=0.05.
\end{equation}
 The numerical integration of the network is realized using a $C++$ program, on  $\Omega\times [0,T] =[0,100]\times[0,100]\times[0,3000]$.
 The main concern of this section are:
 \begin{itemize}
 \item heuristic laws of strength coupling with respect to the number of nodes in the graph
 \item influence of space heterogeneity in initial conditions with regard to the synchronization phenomenon.
 \end{itemize}

Our main conclusion is that heterogeneity in initial conditions does not  affect the general law of synchronization with respect to the number of nodes in the graph. As for ODE's, and in good agreement with our theoretical results, the threshold value for synchronization is given by a $\frac{1}{n}$ law in case of a fully connected network, see figure \ref{FullyCN-Ho-Spi-loi-1-n}  and a $n^2$ law in case of a ring network, see figure \ref{fig:loi-ringuniform} . However, it appears that the space heterogeneity of initial conditions  will increase the threshold value, see figure \ref{ev-FC-PLU-RICunif}. Besides, the persistence of patterns require some symmetry. This kind of question have been mentioned for example in \cite{Ambrosio1,Golubitsky}, and the persistence of patterns, with respect to the size of $\Omega$ and the diffusion coefficient  have been studied in \cite{Ambrosio1, Smoller3}. If there is not a certain kind of symmetry the network evolves towards homogeneous solutions whereas heterogeneous patterns  persist under some symmetry conditions on initial conditions. In the case of one node, this question is theoretically  treated in a forthcoming paper \cite{Ambrosio3}. Indeed, in figure \ref{fig:Nosymmetry-noPatterns}, we show how the increase of coupling strength for non symmetric initial conditions, make the patterns disappear.  Most of our simulations are done with $c(x)=0$ which correspond to the periodic regime of the FHN system. However,  we also perform some numerical simulations with $c(x)$ depending on $x$.  This allows us to take advantage of  both excitability and oscillatory regime of the FHN system. The emerging laws are the same in this case.

We split this section in two subsections corresponding to the fully and unidirectional  connected ring networks.
\subsection{Fully connected network}
For the fully connected network, the system reads as:
\begin{equation}\label{6a6}
\left\{ \begin{gathered}
  \epsilon u_{it}=-u_i^3+3u_i-v_i+d_u\Delta u_i-g_n\sum\limits_{j = 1,j\neq i}^n(u_i-u_j)\hfill \\
    v_{it}=au_i-bv_i+c(x)\hfill \\
\end{gathered}  \right. \,\,\,\,\,\,\,\,\,\,\, i=\overline{1,n},
\end{equation}
where $g_n$ is the constant coupling strength between each couple of nodes in the graph.
We present here five figures.
\begin{itemize}
\item The figure \ref{toibuon1} represents the evolution of the network, for $n=3$, when we choose two initial conditions with spirals and one homogeneous.
The simulations show that we obtain synchronization for a threshold value of  $g_3\simeq 0.015$.
They also show that we obtain asymptotically three spirals patterns.
\item In figure  \ref{FullyCN-Ho-Spi-loi-1-n}, we represent the evolution of the coupling strength value $g_n$  needed for synchronization
with respect to the number of nodes in the graph.
The number of nodes is varied from $3$ to $20$, and we obtain an heuristic law in $\frac{1}{n}$.
This is in good accordance with our theoretical results and is similar with previous results for ODE's networks.
We have also obtained the same heuristic laws when we choose only spirals for initial conditions or only uniform laws on $[-1,1]$ or even with $c(x)$ non-constant function.
Thus the spatial effects induced by initial conditions do not affect this heuristic law already found on ODE's fully linearly coupled networks.
\item In figure \ref{fig:Nosymmetry-noPatterns}, we show a possible effect of increasing the coupling strength.
For some values of the coupling strength, if the initial conditions are near a symmetric configuration a pattern will persist otherwise they asymptotically evolve to a space
homogeneous periodic solution, see \cite{Ambrosio3} and the references therein for a discussion in the case of a single reaction-diffusion system of FHN type.
Here in the case of the fully connected network of FHN systems, for a coupling strength $g_3=0.5$, and with to spirals and one homogeneous for initial conditions,
we observe that the network evolves toward a space homogeneous synchronized state: the system  synchronizes but there is no more patterns.
\item The figure \ref{toibuon8} is obtained in an analogous way as the figure \ref{toibuon1} but with different initial conditions.
Here, for each $x\in \Omega$, we choose $u_i(x,0)$ and $v_i(x,0)$, $i\in \{1,2,3\}$ as  random uniform values in $[-1,1]$.
The simulations show that we obtain synchronization for a threshold value of  $g_3\simeq 0.014$.
They also show that we obtain asymptotically three similar patterns.
\item In figure  \ref{FullyCN-Uniform-loi-1-n}, as in figure \ref{FullyCN-Ho-Spi-loi-1-n}, we represent the evolution of the coupling strength value $g_n$  needed for synchronization with respect to the number of nodes in the graph.
The number of nodes is varied from $3$ to $20$, and we obtain an heuristic law in $\frac{1}{n}$. Thus the spatial effects induced by initial conditions do not affect this heuristic law.
\item In figure \ref{ev-FC-PLU-RICunif}, we show the evolution of the coupling strength value $g_{20}$   needed for synchronization with respect  to the ratio between random uniform
initial conditions and space homogeneous initial conditions in the fully connected network. The figure shows that increasing the space dispersion in initial conditions increase the coupling strength  threshold.
\item Finally, in \ref{figlong1new1}, we show the synchronization of a fully linearly connected network of type \eqref{6a6} with $n=3$, and $c(x)$ a non constant function: $c(x)=0$
if $x$ is in a small neighborhood of $(0,0)$, and $c(x)=-1.1$ overwise.
For one subsystem, and because of the oscillatory-excitable of the FHN system, this induces generation of periodic pulses starting at $(0,0)$.
For initial conditions, we choose three distinct space-homogeneous values. We observe that the synchronization occurs
for $g_3\geq 0.013.$.
Asymptotically, the three subsystems evolve with spiral patterns.

\end{itemize}

\begin{figure}[H]
\begin{flushleft}
\includegraphics[scale=0.5]{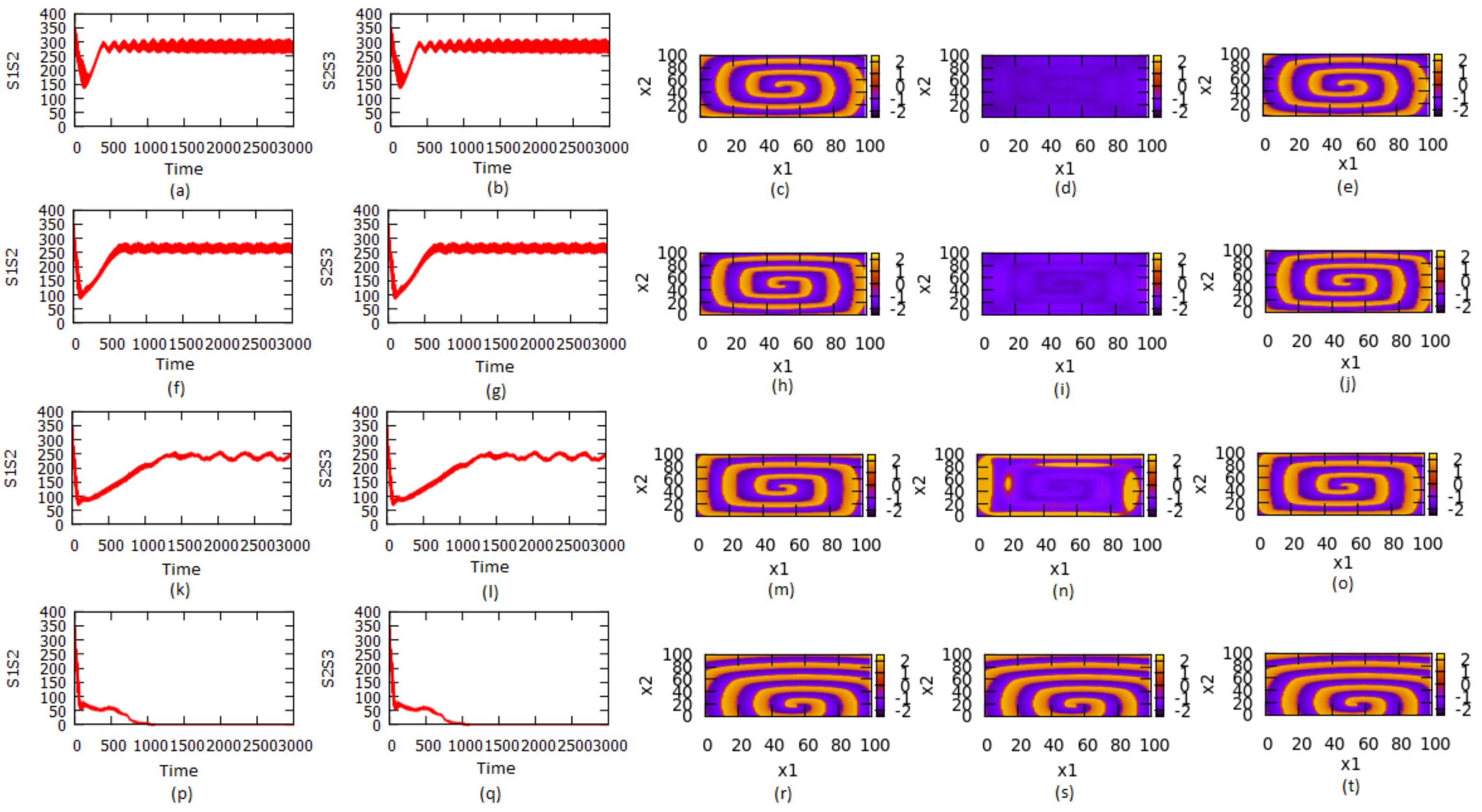}
\caption{Synchronization of a fully linearly connected network of type \eqref{6a6} with $n=3$, $c(x)=0$.
For initial conditions, we choose two spirals and one homogeneous in space. We observe that the synchronization occurs
for $g_3\geq 0.015.$. Indeed, each row, from up to down, correspond respectively to the following values for $g_3$: $0.005, 0.01, 0.013, 0.015$.
The two first columns represent the synchronization error respectively between $u_1$ and $u_2$ and $u_2$ and $u_3$.
The three others columns represent respectively from left to right the isovalues of $u_1$, $u_2$ and $u_3$  for all $x=(x_1,x_2)\in \Omega$ at time $3000$.
Asymptotically, the three subsystems evolve with spiral patterns.}
\label{toibuon1}
\end{flushleft}
\end{figure}

\begin{figure}[H]
\begin{center}
\includegraphics[scale=0.7]{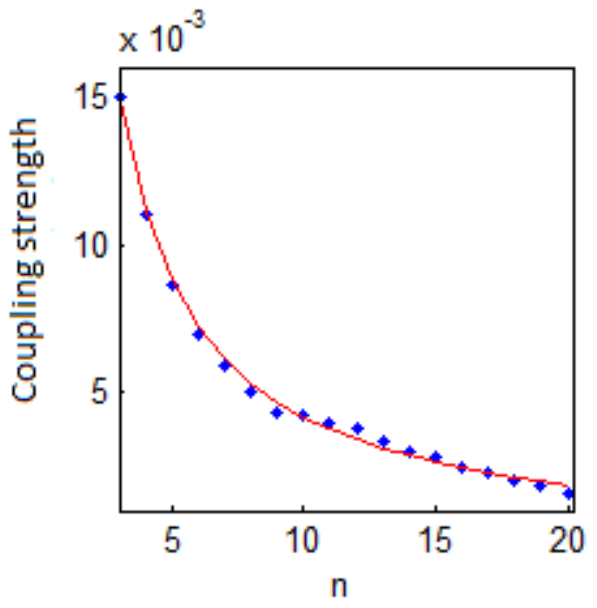}
\begin{quotation}
\caption{ Evolution of the coupling strength value $g_n$  needed for synchronization with the respect to the number of nodes, $n$ in the graph for a fully
connected network of type \eqref{6a6} with $c(x)=0$, when $n$ is varied from $3$ to $20$.
For initial conditions, we choose  approximately $50\%$ of spirals and $50\%$ of homogeneous. The blue points represent the values obtained numerically, the red curve represents the function $\displaystyle g_n=\displaystyle\frac{0.046}{n}-0.00043$. Thus, we obtain heuristically a $\frac{1}{n}$  law which is common in other areas and highlights the synchronization emergent property, see for example \cite{C1, C1p, C2}.}
\label{FullyCN-Ho-Spi-loi-1-n}
\end{quotation}
\end{center}
\end{figure}

\begin{figure}[H]
\begin{center}
\includegraphics[scale=0.5]{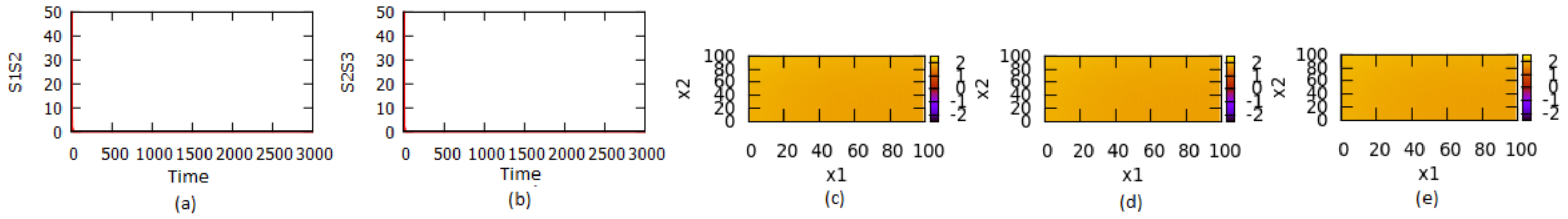}
\begin{quotation}
\caption{Synchronization of a fully linearly connected network of type \eqref{6a6} with $n=3$, $c(x)=0$ and $g_3=0.5$.
For initial conditions, we choose two spirals and one homogeneous in space.
This shows an effect of the increasing of the coupling strength. Indeed, the simulation is the same as in figure \ref{toibuon1} with a coupling strength equal to $0.5$.
We observe that the spiral patterns disappear, and the network evolve to space homogeneous solutions. This comes from a lack of symmetry  in initial solutions}
\label{fig:Nosymmetry-noPatterns}
\end{quotation}

\end{center}
\end{figure}

\begin{figure}[H]
\begin{center}
\includegraphics[scale=0.5]{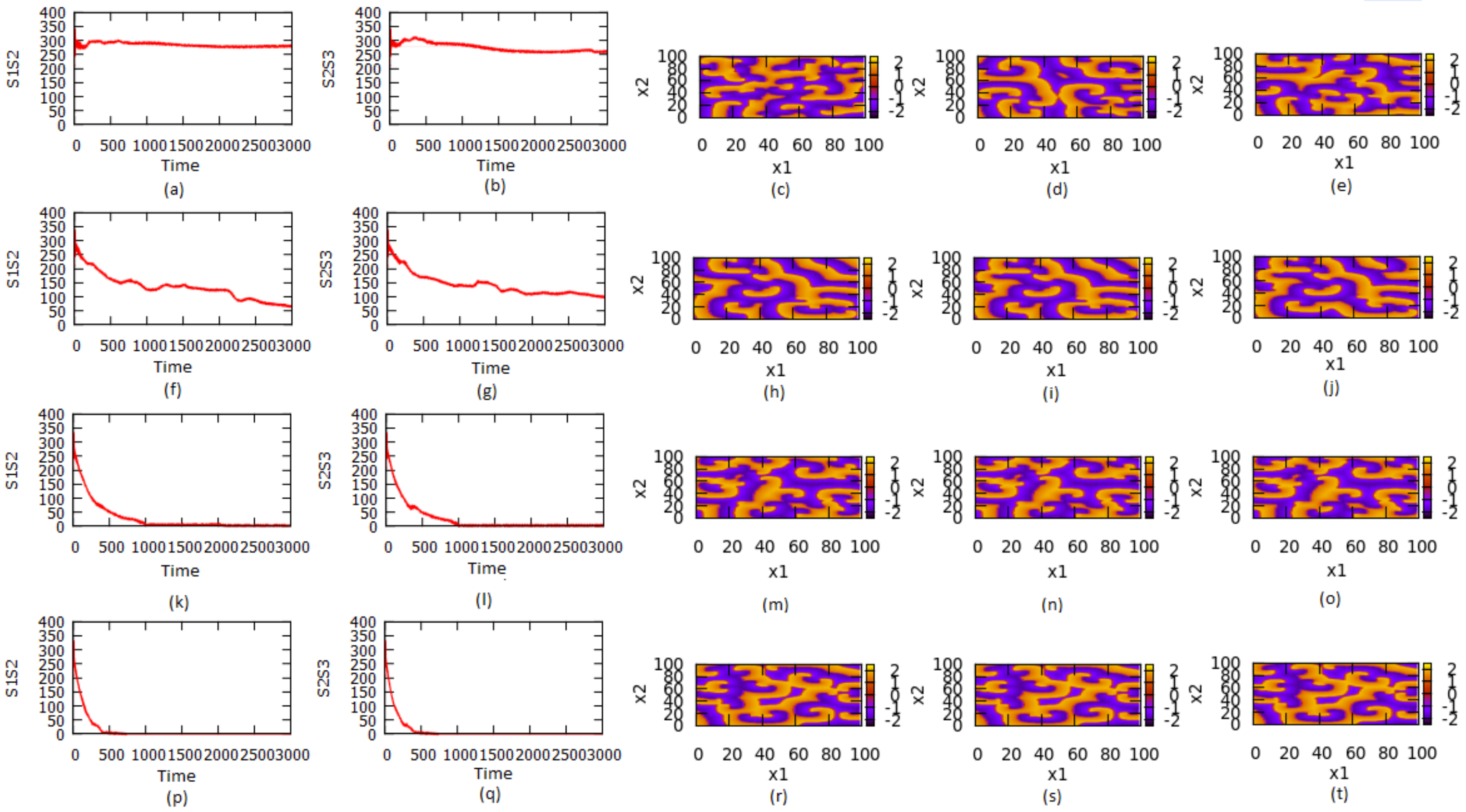}
\begin{quotation}
\caption{Synchronization of a fully linearly connected network of type \eqref{6a6} with $n=3$, $c(x)=0$.
For initial conditions, we choose random initial conditions as follows: for each $x\in \Omega$, we whose $u_i(x,0)$ and $v_i(x,0)$, $i\in \{1,2,3\}$ as  random uniform values in $[-1,1]$.  We observe that the synchronization occurs
for $g_3\geq 0.014.$. Indeed, each row, from up to down, correspond respectively to the following values of $g_3$: $0.001, 0.005, 0.01, 0.014$.
The two first columns represent the synchronization error respectively between $u_1$ and $u_2$ and $u_2$ and $u_3$.
The three others columns represent respectively from left to right the isovalues of $u_1$, $u_2$ and $u_3$  for all $x=(x_1,x_2)\in \Omega$ at time $3000$.
Asymptotically, the three subsystems evolve with the same patterns.}
\label{toibuon8}
\end{quotation}
\end{center}
\end{figure}

\begin{figure}[H]
\begin{center}
\includegraphics[scale=0.7]{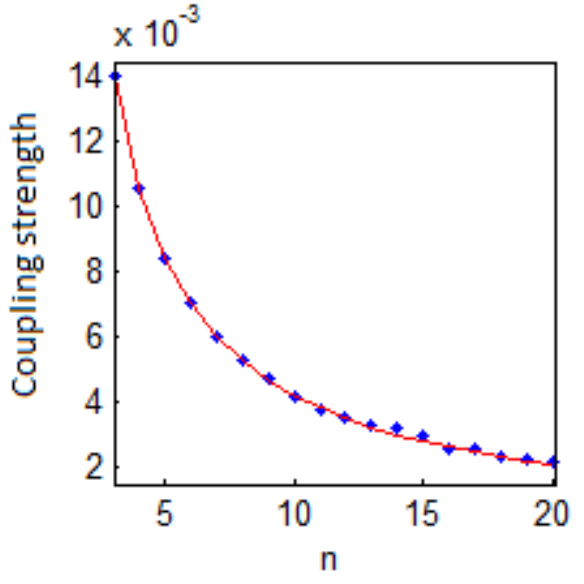}
\begin{quotation}
\caption{ Evolution of the coupling strength value $g_n$  needed for synchronization with the respect to the number of nodes, $n$ in the graph for a fully
connected network of type \eqref{6a6} with $c(x)=0$, when $n$ is varied from $3$ to $20$.
For initial conditions, we choose  random uniform values in $[-1,1]$. The blue points represent the values obtained numerically, the red curve represents the function $\displaystyle g_n=\displaystyle\frac{0.042}{n}$.
Thus, we obtain heuristically a $\frac{1}{n}$ law. Note that this heuristically law is independent of the spatial structure of the initial conditions.}
\label{FullyCN-Uniform-loi-1-n}
\end{quotation}
\end{center}
\end{figure}

\begin{figure}[H]
\begin{center}
\includegraphics[scale=0.9]{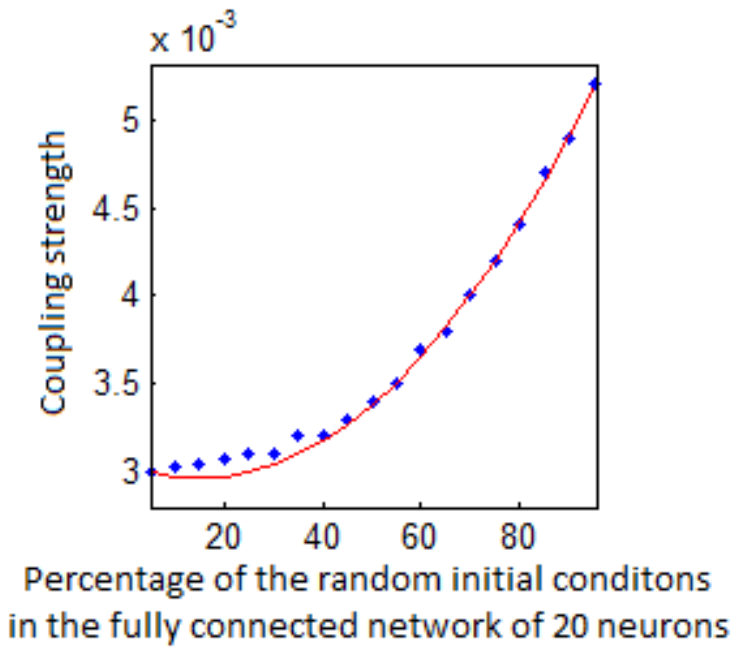}
\begin{quotation}
\caption{\label{ev-FC-PLU-RICunif} Evolution of the coupling strength value $g_{20}$   needed for synchronization with the respect  to the ratio between random uniform
initial conditions and space homogeneous initial conditions in the fully connected network. The blue points represent the values obtained numerically,
the red curve represents the function $\displaystyle g_n(p)=352\times10^{-9}p^2-11\times10^{-6}+0.00305$.}\end{quotation}
\end{center}
\end{figure}

\begin{figure}[H]
\begin{center}
\includegraphics[scale=0.5]{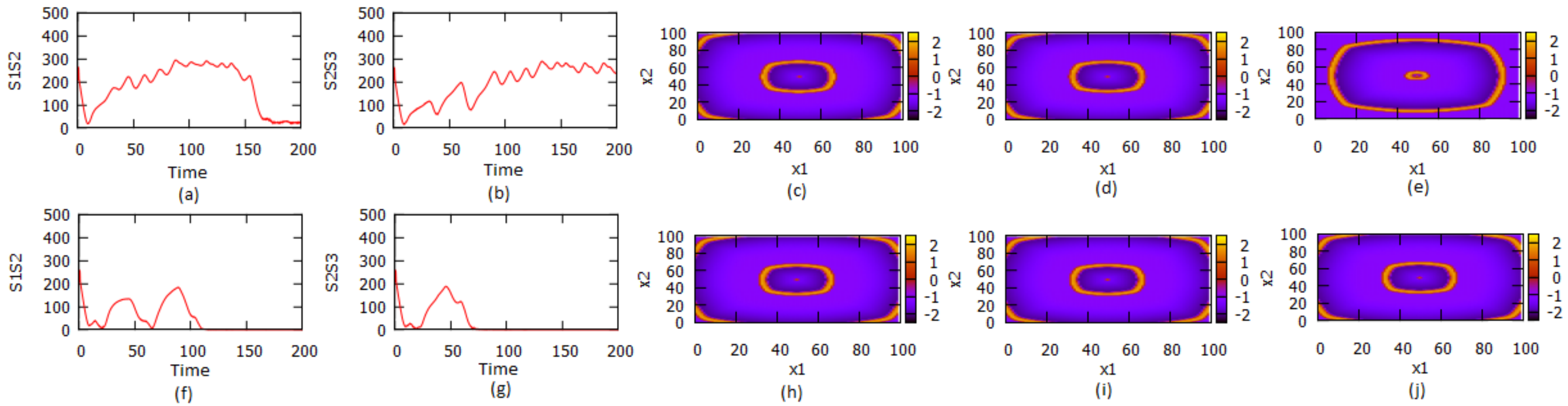}
\begin{quotation}
\caption{Synchronization of a fully linearly connected network of type \eqref{6a6} with $n=3$, and $c(x)=0$ if $x$ is in a small neighborhood of the center of $\Omega$, $c(x)=-1.1$ overwise.
For one subsystem, and because of the oscillatory-excitable of the FHN systeclearpagem,this induces generation o periodic pulses starting at $(0,0)$.
For initial conditions, we choose three distinct space-homogeneous values. We observe that the synchronization occurs
for $g_3\geq 0.013$. Indeed, each row, from up to down, correspond respectively to the following values for $g_3$: $0.01$ and $0.013$.
The two first columns represent the synchronization error respectively between $u_1$ and $u_2$ and $u_2$ and $u_3$.
The three others columns represent respectively from left to right the isovalues of $u_1$, $u_2$ and $u_3$  for all $x=(x_1,x_2)\in \Omega$ at time $3000$.
Asymptotically, the three subsystems evolve with spiral patterns.}
\label{figlong1new1}
\end{quotation}
\end{center}
\end{figure}

\subsection{Unidirectionally coupled ring network}

For the unidirectionally ring connected network, the system reads as:
\begin{equation}\label{eq:ringnetwork}
\left\{ \begin{gathered}
  \epsilon u_{it}=d_u\Delta u_i-u_i^3+3u_i-v_i-g_n\sum\limits_{j = 1,j\neq i}^n(u_i-u_{i+1})\hfill \\
    v_{it}=au_i-bv_i+c(x)\hfill \\
\end{gathered}  \right. \,\,\,\,\,\,\,\,\,\,\, i=\overline{1,n},
\end{equation}
where $g_n$ is the constant coupling strength between each couple of nodes corresponding to an edge in the graph.
As for fully connected networks We present here three figures.
\begin{itemize}
\item The figure \ref{fig:ringallspirals} represent the evolution of the network, for $n=3$, when we choose all initial conditions with spirals.
The simulations show that we obtain synchronization for a threshold value of  $g_3\simeq 0.001$.
They also show that we obtain asymptotically three spirals patterns.
\item The figure \ref{fig:ringuniform} is obtained in an analogous way as the figure \ref{fig:ringallspirals} but with different initial conditions.
Here, for each $x\in \Omega$, we choose $u_i(x,0)$ and $v_i(x,0)$, $i\in \{1,2,3\}$ as  random uniform values in $[-1,1]$.
The simulations show that we obtain synchronization for a threshold value of  $g_3\simeq 0.02$.
They also show that we obtain asymptotically three similar patterns.
\item In figure  \ref{fig:loi-ringuniform}, we represent the evolution of the coupling strength value $g_n$  needed for synchronization with respect to the number of nodes in the graph.
The number of nodes is varied from $3$ to $20$, and we obtain an heuristic law in $n^2$. This is in good accordance with our theoretical results and with previous results for ODE's networks.
We have also obtained the same heuristic laws when we choose only spirals for initial conditions or spirals and homogeneous or even with $c(x)$ non-constant function.
Thus the spatial effects induced by initial conditions do not affect this heuristic law already found on ODE's ring linearly coupled networks.

\end{itemize}

\begin{figure}[H]
\begin{center}
\includegraphics[scale=0.5]{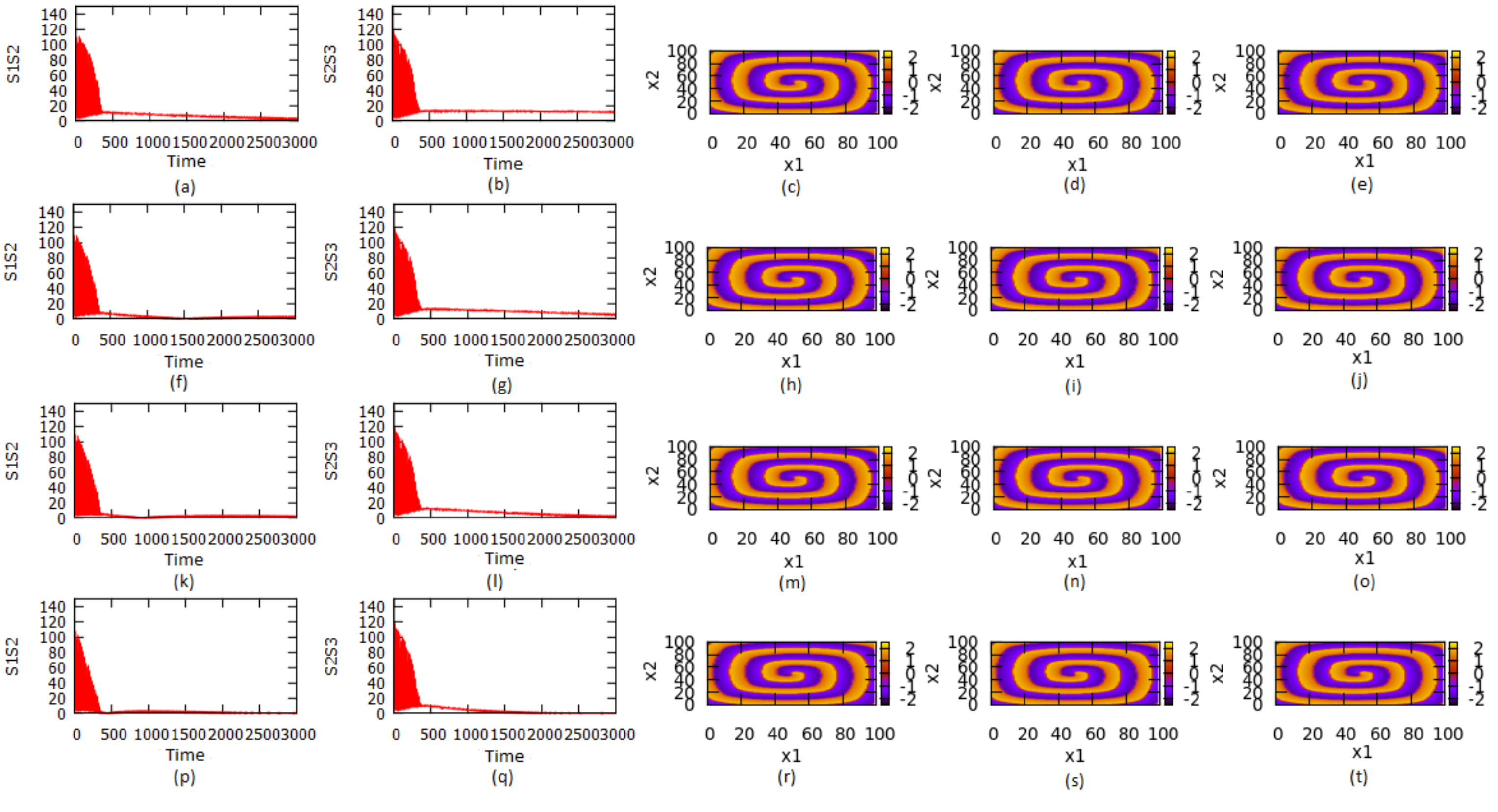}
\begin{quotation}
\caption{Synchronization of a fully linearly connected network of type \eqref{eq:ringnetwork} with $n=3$, $c(x)=0$.
For initial conditions, we choose three spirals. We observe that the synchronization occurs
for $g_3\geq 0.001$. Indeed, each row, from up to down, correspond respectively to the following values for $g_3$: $0.0003, 0.0005, 0.001$.
The two first columns represent the synchronization error respectively between $u_1$ and $u_2$ and $u_2$ and $u_3$.
The three others columns represent respectively from left to right the isovalues of $u_1$, $u_2$ and $u_3$  for all $x=(x_1,x_2)\in \Omega$ at time $3000$.
Asymptotically, the three subsystems evolve with spiral patterns.}
\label{fig:ringallspirals}
\end{quotation}
\end{center}
\end{figure}

\begin{figure}[H]
\begin{center}
\includegraphics[scale=0.5]{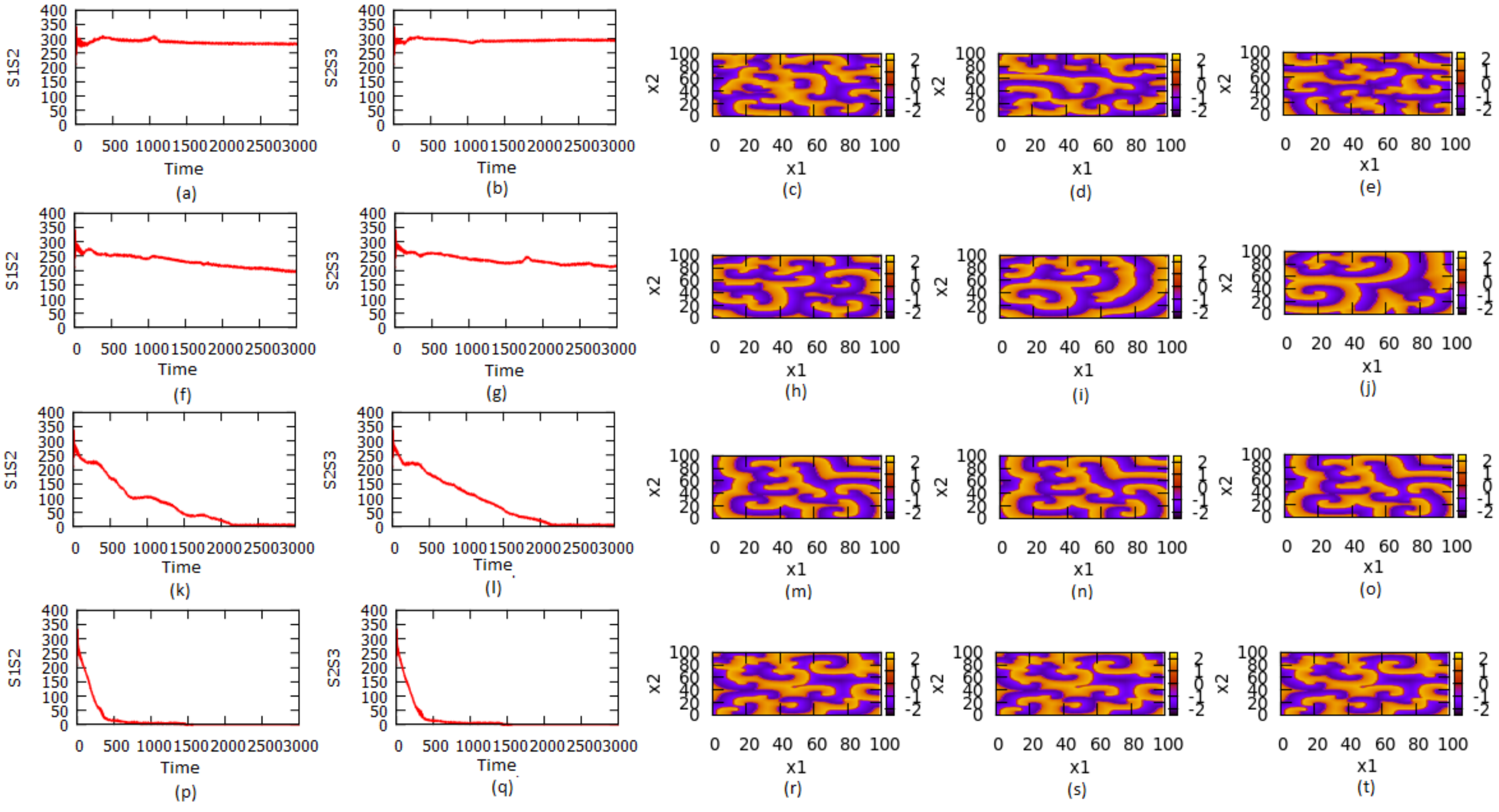}
\begin{quotation}
\caption{Synchronization of a fully linearly connected network of type \eqref{eq:ringnetwork} with $n=3$, $c(x)=0$.
For initial conditions, we choose random initial conditions as follows: for each $x\in \Omega$, we whose $u_i(x,0)$ and $v_i(x,0)$, $i\in \{1,2,3\}$ as  random uniform values in $[-1,1]$.
We observe that the synchronization occurs
for $g_3\geq 0.02$. Indeed, each row, from up to down, correspond respectively to the following values of $g_3$: $0.001, 0.005, 0.02$.
The two first columns represent the synchronization error respectively between $u_1$ and $u_2$ and $u_2$ and $u_3$.
The three others columns represent respectively from left to right the isovalues of $u_1$, $u_2$ and $u_3$  for all $x=(x_1,x_2)\in \Omega$ at time $3000$.
Asymptotically, the three subsystems evolve with the same patterns.}
\label{fig:ringuniform}
\end{quotation}
\end{center}
\end{figure}

\begin{figure}[H]
\begin{center}
\includegraphics[scale=0.7]{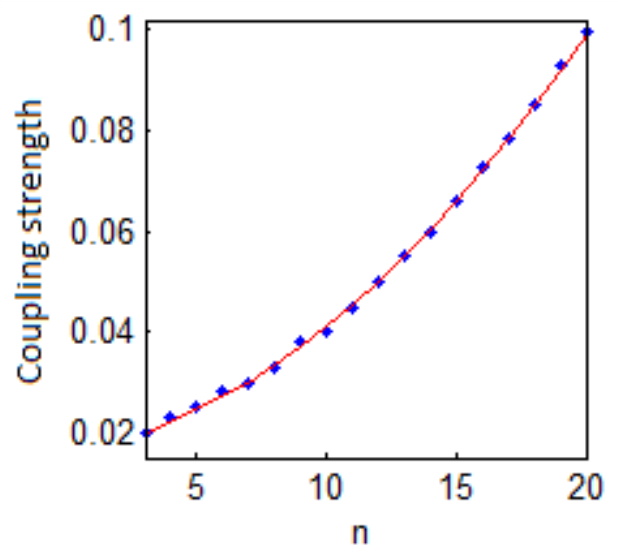}
\caption{Evolution of the coupling strength value $g_n$  needed for synchronization with the respect to the number of nodes, $n$ in the graph of the unidirectionally ring
connected network of type \eqref{eq:ringnetwork} with $c(x)=0$, when $n$ is varied from $3$ to $20$.
For initial conditions, we choose  approximately $50\%$ of spirals and $50\%$ of homogeneous. The blue points represent the values obtained numerically, the red curve represents
the function $\displaystyle  g_n=0.0000167n^2+0.00062n+0.02$. }
\label{fig:loi-ringuniform}
\end{center}
\end{figure}

\section{Conclusion} In this paper, we have considered a network of $n$ coupled  reaction-diffusion systems. We obtained three main contributions. First, we prove of the existence of the network attractor, and therefore, within this attractor, we have analyzed the synchronization behavior. We found out theoretically threshold values for synchronization for general class of networks with linear coupling. Finally, we have performed numerical simulations with different kind
of initial conditions and studied numerically the spatial effects on synchronization and pattern formation. Thanks for the numerical experiments, we exhibited heuristic laws  with regard to the number of nodes in the graph. The general form of these heuristic laws does not depend on  spatial  heterogeneity of asymptotic behavior. Our numerical simulations also show the persistence of asymptotic spatial patterns in complex networks. In future work, we aim to apply these results to more realistic neuronal networks.

\appendix
\textbf{Appendix}
\begin{theorem}
\label{th:GrowUni}
Let $g$,$h$ and $y\in L^1(\R)$ three positive functions. We suppose that for all $t \geq t_0$:
\begin{equation}
\label{ineq:gy}
\frac{dy}{dt}\leq gy+h
\end{equation}
and,
\begin{equation}
\int_t^{t+r}g(s)ds\leq a_1,\int_t^{t+r}h(s)ds\leq a_2,\int_t^{t+r}y(s)ds\leq a_3,
\end{equation}
where $r,a_1,a_2,a_3$ are positive constants. Then
\begin{equation}
y(t+r)\leq(a_3r+a_2)e^{a_1},\forall t \geq t_0.
\end{equation}
\end{theorem}
\begin{proof}
Let $s_0\geq t_0$. By \eqref{ineq:gy}, we have:
\begin{equation}
\frac{d}{dt}\big{(}e^{-\int_{s_0}^tg(s)ds}y(t)\big{)}\leq e^{-\int_{s_0}^tg(s)ds}h.
\label{eq:demlemGrU}
\end{equation}
We integrate \eqref{eq:demlemGrU} between $t$ and $s_0+r$ for $t \in [s_0,s_0+r]$. We obtain:
\begin{equation}
e^{-\int_{s_0}^{s_0+r}g(s)ds}y(s_0+r)-e^{-\int_{s_0}^{t}g(s)ds}y(t) \leq \int_t^{s_0+r} \exp\{-\int_{s_0}^{t'} g(s)ds\}h(t')dt'.
\end{equation}
It follows that:
\begin{equation}
e^{-a_1}y(s_0+r)-e^{-\int_{s_0}^{t}g(s)ds}y(t)\leq\int_t^{s_0+r} h(t')dt'.
\end{equation}
By multiplying by $\exp\{\int_{s_0}^{t}g(s)ds\}$, we find:
\begin{equation}
e^{\int_{s_0}^{t}g(s)}e^{-a_1}y(s_0+r)-y(t)\leq e^{\int_{s_0}^{t}g(s)ds}\int_t^{s_0+r} h(t')dt'.
\end{equation}
Then we integrate between $s_0$ and $s_0+r$ with respect to $t$. This gives:
\begin{equation}
e^{-a_1}y(s_0+r)\leq a_3 +a_2r,
\end{equation}
thus:
\begin{equation}
y(s_0+r)\leq (a_3 +a_2r)e^{a_1}.
\end{equation}
\end{proof}

\begin{corollary}
Let  $y, h\in L^1(\R)$ two positive functions. We assume that fort $t \geq t_0$:
\begin{equation}
\label{ineq:yprime}
\frac{dy}{dt}\leq h
\end{equation}
and,
\begin{equation}
\int_t^{t+r}h(s)ds \leq a_2,\ \int_t^{t+r}ds \leq a_3,
\end{equation}
where $r,a_2,a_3$ are positive constants. Then
\begin{equation}
y(t+r)\leq \frac{a_3}{r}+a_2,\forall t \geq t_0.
\end{equation}
\end{corollary}
\begin{proof}
It follows obviously from theorem \ref{th:GrowUni} with $a_1=0$. We can also give a direct proof:
\begin{equation}
y(s_0+r)-y(t)\leq \int_t{s_0+r}hds.
\end{equation}
Then we integrate between $s_0$ et $s_0+r$ with respect to $t$. We obtain,
\begin{equation}
y(s_0+r)\leq \frac{a_3}{r}+a_2.
\end{equation}

\end{proof}

\end{document}